\documentclass[a4j,reqno]{amsart}
\usepackage{amsmath,amsthm}
\usepackage{amssymb}
\usepackage{amsthm}
\usepackage{amscd}
\usepackage{stmaryrd}
\usepackage[dvipdfmx]{graphicx}
\usepackage[all]{xy}
\usepackage{wrapfig}
\usepackage{color}
\usepackage{latexsym}
\usepackage{enumerate}
\usepackage{cases}
\usepackage[top=30mm,bottom=30mm,left=30mm,right=30mm]{geometry}
\usepackage{graphics}
\numberwithin{equation}{section}
\theoremstyle{definition}
\newtheorem{example}{Example}[section]
\newtheorem{definition}[example]{Definition}

\newtheorem{lemma}[example]{Lemma}
\newtheorem{theorem}[example]{Theorem}
\newtheorem{proposition}[example]{Proposition}
\newtheorem{corollary}[example]{Corollary}

\def\MARU#1{\leavevmode \setbox0\hbox{$\bigcirc$}%
 \copy0\kern-\wd0 \hbox to\wd0{\hfil{#1}\hfil}}

\title{Cluster expansion formulas in type $A$}
\author{toshiya yurikusa}
\address{T. Yurikusa: Graduate School of Mathematics, Nagoya University, Chikusa-ku, Nagoya, 464-8602 Japan}
\email{m15049q@math.nagoya-u.ac.jp}

\begin{document}

\maketitle

\begin{abstract}
 The aim of this paper is to give analogs of the cluster expansion formula of Musiker and Schiffler for cluster algebras of type $A$ with coefficients arising from boundary arcs of the corresponding triangulated polygon. Indeed, we give three cluster expansion formulas by perfect matchings of angles in triangulated polygon, by discrete subsets of arrows of the corresponding ice quiver and by minimal cuts of the corresponding quiver with potential.
\end{abstract}


\section{Introduction}

\subsection{Background}
 Cluster algebras introduced by Fomin and Zelevinsky in $2002$ \cite{FZ1} are commutative algebras with a distinguished set of generators, which are called cluster variables. Their original motivation was coming from an algebraic framework for total positivity and canonical bases in Lie Theory. In recent years, it has applied to various subjects in mathematics, for example, quiver representations, Calabi-Yau categories, Poisson geometry, Teichm\"{u}ller spaces, exact WKB analysis, etc. We refer to the introductory articles \cite{FZ1,GSV,K,S} for details.

 By Laurent phenomenon, any cluster variable in a cluster algebra is expressed by a Laurent polynomial of the initial cluster variables $(x_1,\ldots,x_n)$ amd coefficients $(x_{n+1},\ldots,x_m)$
 \begin{equation*}
  x=\frac{f(x_1, \ldots,x_m)}{x_1^{d_1}x_2^{d_2} \cdots x_n^{d_n}}
 \end{equation*}
where $f(x_1, \ldots, x_m) \in \mathbb{Z}[x_1,\ldots, x_m]$ and $d_i \in \mathbb{Z}_{\ge 0}$ \cite{FZ1}. An explicit formula for the Laurent polynomials of cluster variables is called a cluster expansion formula. In \cite{LS1}, Lee and Schiffler gave a cluster expansion formula for the cluster algebras of rank $2$ in terms of Dyck paths, and as an application, they proved in \cite{LS2} the positivity conjecture of Fomin and Zelevinsky \cite{FZ1}. Another well-known cluster expansion formula is Caldero-Chapoton formula given by categorification \cite{CC,CK,P,Pl}. It was also applied to solve several conjectures of Fomin and Zelevinsky (see e.g. \cite{CKLP}).

 It was posed in \cite{LS2} as an open problem to find out explicit formulas for the cluster variables in any cluster algebra. An example of such formulas is given by Musiker and Schiffler \cite{MS} for the cluster algebras associated with triangulated surfaces ($S,M$) with marked points without punctures studied by Fomin, Shapiro and Thurston \cite{FST}. Their formula is described in terms of the perfect matchings in a labeled graph $G$ constructed from $(S,M)$. It was extended to cluster algebras associated with triangulated surfaces with punctures by Musiker, Schiffler and Williams {\cite{MSW}}.\par


\subsection{Our results}
 In this paper we study cluster algebras of type $A_n$ with $n+3$ coefficients. More precisely, it is the cluster algebra structure of the homogeneous coordinate ring of the Grassmannian ${\rm Gr}(2,n+3)$ introduced in \cite{FZ2}.

 We denote by $[1,n]$ the interval $\{1,2,\ldots,n\}$. For a triangulated $(n+3)$-gon $T$ with $n$ diagonals, we construct a quiver $Q_T$ using the following rule \cite{CCS,FST}.\\
 $\bullet$ If the diagonals $i$ and $j$ are in a common triangle of $T$ and $j$ follows $i$ in the clockwise order, then there is an arrow from $i$ to $j$:
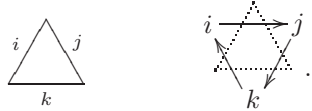
\[ \begin{xy}
   (0,-2)="0", +(5,8.7)="1", +(5,-8.7)="2"
   \ar@{-}"0";"1"^{i} \ar@{-}"1";"2"^{j} \ar@{-}"2";"0"^{k}
  \end{xy}\hspace{15mm}
  \begin{xy}
   (0,0)="0", +(5,8.7)="1", +(5,-8.7)="2", "0"+(-1,6)="i"*{i}, "2"+(1,6)="j"*{j}, "0"+(5,-4)="k"*{k}
   \ar@{.}"0";"1" \ar@{.}"1";"2" \ar@{.}"2";"0" \ar@{->}"i"+(1.5,0);"j"+(-1.5,0) \ar@{->}"j"+(-1,-2);"k"+(1.5,1.5) \ar@{->}"k"+(-1.5,1.5);"i"+(1,-2)
  \end{xy}.\]
 We define a new quiver $\overline{Q}_T$ from $Q_T$ by adding a vertex for each boundary arc of $T$ and arrows between original vertices and new vertices following the above rule. We give an example of this construction:
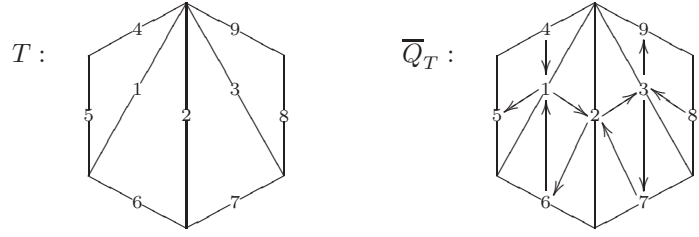
\[
  T : \hspace{4mm}\begin{xy}
   (0,0)="0", +(13,7)="1", +(13,-7)="2", +(0,-16)="3", +(-13,-7)="4", +(-13,7)="5"
   \ar@{-}"0";"1"|{4} \ar@{-}"1";"2"|{9} \ar@{-}"2";"3"|{8} \ar@{-}"3";"4"|{7} \ar@{-}"4";"5"|{6} \ar@{-}"5";"0"|{5} \ar@{-}"1";"5"|{1} \ar@{-}"1";"4"|{2} \ar@{-}"1";"3"|{3}
  \end{xy}\hspace{15mm}
 \overline{Q}_T : \hspace{4mm}\begin{xy}
   (0,0)="0", +(13,7)="1", +(13,-7)="2", +(0,-16)="3", +(-13,-7)="4", +(-13,7)="5"
   \ar@{-}"0";"1"|{4} \ar@{-}"1";"2"|{9} \ar@{-}"2";"3"|{8} \ar@{-}"3";"4"|{7} \ar@{-}"4";"5"|{6} \ar@{-}"5";"0"|{5} \ar@{-}"1";"5"|{1} \ar@{-}"1";"4"|{2} \ar@{-}"1";"3"|{3} \ar@{->}"1"+(-6.5,-5);"1"+(-6.5,-9.5) \ar@{->}"1"+(-7.5,-11.5);"1"+(-12,-14.5) \ar@{->}"1"+(-5.5,-12);"1"+(-1,-15) \ar@{->}"1"+(-1,-16);"1"+(-5.5,-25.5) \ar@{->}"1"+(-6.5,-25);"1"+(-6.5,-13) \ar@{->}"1"+(1,-15);"1"+(5.5,-12) \ar@{->}"1"+(6.5,-13);"1"+(6.5,-25) \ar@{->}"1"+(5.5,-25.5);"1"+(1,-16) \ar@{->}"1"+(6.5,-9.5);"1"+(6.5,-5) \ar@{->}"1"+(12,-14.5);"1"+(7.5,-11.5)
  \end{xy}
\]

 Recall that any quiver of type $A_n$ is isomorphic to $Q_T$ for some triangulated $(n+3)$-gon $T$ \cite{CCS}. Let $T$ be a triangulated $(n+3)$-gon such that $Q:=Q_T$ is an acyclic quiver with underlying graph

\[
  \begin{xy}
   (0,0)="1"*{1}, +(10,0)="2"*{2}, +(10,0)="3"*{3}, +(10,0)="A"*{}, +(10,0)="B"*{}, +(10,0)="n"*{n}
   \ar@{-}"1"+(2,0);"2"+(-2,0) \ar@{-}"2"+(2,0);"3"+(-2,0) \ar@{-}"3"+(2,0);"A"+(-2,0) \ar@{.}"A"+(2,0);"B"+(-2,0) \ar@{-}"B"+(2,0);"n"+(-2,0)
  \end{xy}.
\]
 We regard $\overline{Q} := \overline{Q}_T$ as an ice quiver of type $(n,2n+3)$ with frozen vertices $[n+1,2n+3]$ and denote by ${\mathcal A}(\overline{Q})$ the corresponding cluster algebra (see Definition \ref{clusteralg}). It is known that non-initial cluster variables of ${\mathcal A}(\overline{Q})$ are indexed by pairs $(i,j)$ of integers such that $1 \le i \le j \le n$ \cite{FZ2,ST}. More precisely, $(i,j)$ is associated with a cluster variable
 \begin{equation*}
  \frac{f^{[i,j]}}{x_ix_{i+1} \cdots x_j},
 \end{equation*}
where $f^{[i,j]} \in \mathbb{Z}[x_1,\ldots,x_{2n+3}]$ is not divisible by $x_1,\ldots,x_n$.

 Our main result is the cluster expansion formula for $f^{[i,j]}$ using the following notion.

\begin{definition}\label{perangle}
 A {\it perfect matching} of angles in $T$ is a selection of one marked angle per vertex incident to at least one diagonal such that each triangle of $T$ has exactly one marked angle. We denote by ${\mathbb A}(T)$ the set of all perfect matchings of angles in $T$.
\end{definition}

 Clearly, any perfect matching $A \in {\mathbb A}(T)$ satisfies $|A|=n+1$. For example, we provide the complete list of ${\mathbb A}(T)$ for the above case $T$:
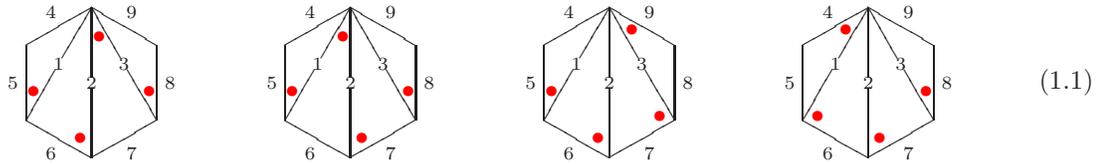
 \begin{equation}\label{firstpmex}
  \begin{xy}
   (0,5)="0", +(8.7,5)="1", +(8.7,-5)="2", +(0,-10)="3", +(-8.7,-5)="4", +(-8.7,5)="5", "1"+(1,-4)*{{\color{red}\bullet}}, "5"+(1,3.8)*{{\color{red}\bullet}}, "4"+(-1.5,2.5)*{{\color{red}\bullet}}, "3"+(-1,3.8)*{{\color{red}\bullet}}
   \ar@{-}"0";"1"^{4} \ar@{-}"1";"2"^{9} \ar@{-}"2";"3"^{8} \ar@{-}"3";"4"^{7} \ar@{-}"4";"5"^{6} \ar@{-}"5";"0"^{5} \ar@{-}"1";"5"|{1} \ar@{-}"1";"4"|{2} \ar@{-}"1";"3"|{3}
  \end{xy}
  \hspace{10mm}\begin{xy}
   (0,5)="0", +(8.7,5)="1", +(8.7,-5)="2", +(0,-10)="3", +(-8.7,-5)="4", +(-8.7,5)="5", "1"+(-1,-4)*{{\color{red}\bullet}}, "5"+(1,3.8)*{{\color{red}\bullet}}, "4"+(1.5,2.5)*{{\color{red}\bullet}}, "3"+(-1,3.8)*{{\color{red}\bullet}}
   \ar@{-}"0";"1"^{4} \ar@{-}"1";"2"^{9} \ar@{-}"2";"3"^{8} \ar@{-}"3";"4"^{7} \ar@{-}"4";"5"^{6} \ar@{-}"5";"0"^{5} \ar@{-}"1";"5"|{1} \ar@{-}"1";"4"|{2} \ar@{-}"1";"3"|{3}
  \end{xy}
  \hspace{10mm}\begin{xy}
   (0,5)="0", +(8.7,5)="1", +(8.7,-5)="2", +(0,-10)="3", +(-8.7,-5)="4", +(-8.7,5)="5", "1"+(3,-3)*{{\color{red}\bullet}}, "5"+(1,3.8)*{{\color{red}\bullet}}, "4"+(-1.5,2.5)*{{\color{red}\bullet}}, "3"+(-2,0.5)*{{\color{red}\bullet}}
   \ar@{-}"0";"1"^{4} \ar@{-}"1";"2"^{9} \ar@{-}"2";"3"^{8} \ar@{-}"3";"4"^{7} \ar@{-}"4";"5"^{6} \ar@{-}"5";"0"^{5} \ar@{-}"1";"5"|{1} \ar@{-}"1";"4"|{2} \ar@{-}"1";"3"|{3}
  \end{xy}
  \hspace{10mm}\begin{xy}
   (0,5)="0", +(8.7,5)="1", +(8.7,-5)="2", +(0,-10)="3", +(-8.7,-5)="4", +(-8.7,5)="5", "1"+(-3,-3)*{{\color{red}\bullet}}, "5"+(2,0.5)*{{\color{red}\bullet}}, "4"+(1.5,2.5)*{{\color{red}\bullet}}, "3"+(-1,3.8)*{{\color{red}\bullet}}
   \ar@{-}"0";"1"^{4} \ar@{-}"1";"2"^{9} \ar@{-}"2";"3"^{8} \ar@{-}"3";"4"^{7} \ar@{-}"4";"5"^{6} \ar@{-}"5";"0"^{5} \ar@{-}"1";"5"|{1} \ar@{-}"1";"4"|{2} \ar@{-}"1";"3"|{3}
  \end{xy}
 \end{equation}

 We denote by $T^{[i,j]}$ the triangulated subpolygon of $T$ such that the set of diagonals of $T^{[i,j]}$ is $[i,j]$. We give our main result using perfect matchings of angles in $T^{[i,j]}$.

\begin{theorem}\label{angle}
 For $1 \le i \le j \le n$, we have
 \begin{equation*}
  f^{[i,j]}=\sum_{A \in {\mathbb A}(T^{[i,j]})} \prod_{a \in A}x_{a},
 \end{equation*}
where $x_a$ is the initial cluster variable corresponding to the opposite side of the angle $a$.
\end{theorem}

 In the above example (\ref{firstpmex}), we obtain $f^{[1,3]}=x_1x_4x_7x_9+x_3x_4x_6x_9+x_1x_2x_4x_8+x_2x_3x_5x_9$ by Theorem \ref{angle}. Also, we obtain $f^{[1,2]}=x_1x_4x_7+x_3x_4x_6+x_2x_3x_5$ since the complete list of ${\mathbb A}(T^{[1,2]})$ is the following:
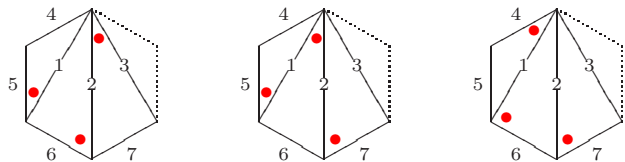
\[
  \begin{xy}
   (0,5)="0", +(8.7,5)="1", +(8.7,-5)="2", +(0,-10)="3", +(-8.7,-5)="4", +(-8.7,5)="5", "1"+(1,-4)*{{\color{red}\bullet}}, "5"+(1,3.8)*{{\color{red}\bullet}}, "4"+(-1.5,2.5)*{{\color{red}\bullet}}
   \ar@{-}"0";"1"^{4} \ar@{.}"1";"2" \ar@{.}"2";"3" \ar@{-}"3";"4"^{7} \ar@{-}"4";"5"^{6} \ar@{-}"5";"0"^{5} \ar@{-}"1";"5"|{1} \ar@{-}"1";"4"|{2} \ar@{-}"1";"3"|{3}
  \end{xy}
 \hspace{10mm}\begin{xy}
   (0,5)="0", +(8.7,5)="1", +(8.7,-5)="2", +(0,-10)="3", +(-8.7,-5)="4", +(-8.7,5)="5", "1"+(-1,-4)*{{\color{red}\bullet}}, "5"+(1,3.8)*{{\color{red}\bullet}}, "4"+(1.5,2.5)*{{\color{red}\bullet}}
   \ar@{-}"0";"1"^{4} \ar@{.}"1";"2" \ar@{.}"2";"3" \ar@{-}"3";"4"^{7} \ar@{-}"4";"5"^{6} \ar@{-}"5";"0"^{5} \ar@{-}"1";"5"|{1} \ar@{-}"1";"4"|{2} \ar@{-}"1";"3"|{3}
  \end{xy}
 \hspace{10mm}\begin{xy}
   (0,5)="0", +(8.7,5)="1", +(8.7,-5)="2", +(0,-10)="3", +(-8.7,-5)="4", +(-8.7,5)="5", "1"+(-3,-3)*{{\color{red}\bullet}}, "5"+(2,0.5)*{{\color{red}\bullet}}, "4"+(1.5,2.5)*{{\color{red}\bullet}}
   \ar@{-}"0";"1"^{4} \ar@{.}"1";"2" \ar@{.}"2";"3" \ar@{-}"3";"4"^{7} \ar@{-}"4";"5"^{6} \ar@{-}"5";"0"^{5} \ar@{-}"1";"5"|{1} \ar@{-}"1";"4"|{2} \ar@{-}"1";"3"|{3}
  \end{xy}
\]

 Next, we give another cluster expansion formula in type $A_n$ in terms of quivers. We introduce {\it discrete subsets} of an ice quiver.

\begin{definition}\label{discrete}
 Let $R=(R_0,R_1,s,t)$ be an ice quiver, and $R'$ the full subquiver of $R$ consisting of all non-frozen vertices. A subset $D$ of $R_1$ is called {\it discrete} if for any $\alpha,\beta \in D$ there is no path in $R'$ from $t(\alpha)$ to $s(\beta)$. We denote by ${\mathbb D}(R)$ the set of all maximal discrete subsets of $R$.
\end{definition}

 We denote by $\overline{Q}^{[i,j]}$ the ice quiver obtained from $T^{[i,j]}$.

\begin{corollary}\label{main1}
 For $1 \le i \le j \le n$, we have
 \begin{equation*}
  f^{[i,j]}=\sum_{D \in {\mathbb D}(\overline{Q}^{[i,j]})} \prod_{\alpha \in D}x_{\alpha},
 \end{equation*}
where $x_{\alpha}$ is the initial cluster variable corresponding to the third side of the triangle in $T$ with sides $s(\alpha)$ and $t(\alpha)$.
\end{corollary}

 We prove Corollary \ref{main1} by giving a natural bijection between ${\mathbb A}(T^{[i,j]})$ and ${\mathbb D}(\overline{Q}^{[i,j]})$ (see Proposition \ref{DPM}). For example, for the case $Q = [ 1 \rightarrow 2 \rightarrow 3 ]$, ${\mathbb A}(T^{[1,2]})$ and ${\mathbb D}(\overline{Q}^{[1,2]})$ are the following:
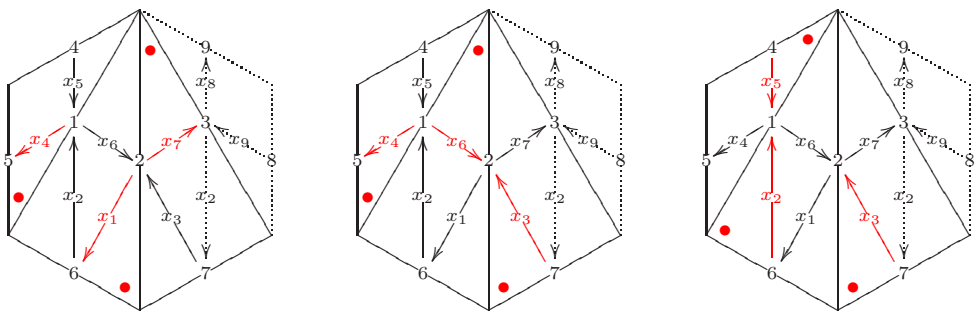
\[
  \begin{xy}
   (0,0)="0", +(17.5,10)="1", +(17.5,-10)="2", +(0,-20)="3", +(-17.5,-10)="4", +(-17.5,10)="5", "5"+(1.4,5)*{{\color{red}\bullet}}, "4"+(-2,3)*{{\color{red}\bullet}}, "1"+(1.4,-5.5)*{{\color{red}\bullet}}
   \ar@{-}"0";"1"|{4} \ar@{.}"1";"2"|{9} \ar@{.}"2";"3"|{8} \ar@{-}"3";"4"|{7} \ar@{-}"4";"5"|{6} \ar@{-}"5";"0"|{5} \ar@{-}"1";"5"|{1} \ar@{-}"1";"4"|{2} \ar@{-}"1";"3"|{3} \ar@{->}"1"+(-8.8,-6.5);"1"+(-8.8,-13)|{x_5} {\color{red}\ar@{->}"1"+(-10,-15.5);"1"+(-16.5,-19.5)|{x_4}}\ar@{->}"1"+(-7.5,-15.5);"1"+(-1,-20)|{x_6} {\color{red}\ar@{->}"1"+(-1,-22);"1"+(-7.5,-33.5)|{x_1}}\ar@{->}"1"+(-8.8,-33);"1"+(-8.8,-17)|{x_2} {\color{red}\ar@{->}"1"+(1,-20);"1"+(7.5,-15.5)|{x_7}}\ar@{.>}"1"+(8.8,-17);"1"+(8.8,-33)|{x_2} \ar@{->}"1"+(7.5,-33.5);"1"+(1,-22)|{x_3} \ar@{.>}"1"+(8.8,-13);"1"+(8.8,-6.5)|{x_8} \ar@{.>}"1"+(16.5,-19.5);"1"+(10,-15.5)|{x_9}
  \end{xy}
 \hspace{10mm}
  \begin{xy}
   (0,0)="0", +(17.5,10)="1", +(17.5,-10)="2", +(0,-20)="3", +(-17.5,-10)="4", +(-17.5,10)="5", "5"+(1.4,5)*{{\color{red}\bullet}}, "4"+(2,3)*{{\color{red}\bullet}}, "1"+(-1.4,-5.5)*{{\color{red}\bullet}}
   \ar@{-}"0";"1"|{4} \ar@{.}"1";"2"|{9} \ar@{.}"2";"3"|{8} \ar@{-}"3";"4"|{7} \ar@{-}"4";"5"|{6} \ar@{-}"5";"0"|{5} \ar@{-}"1";"5"|{1} \ar@{-}"1";"4"|{2} \ar@{-}"1";"3"|{3} \ar@{->}"1"+(-8.8,-6.5);"1"+(-8.8,-13)|{x_5} {\color{red}\ar@{->}"1"+(-10,-15.5);"1"+(-16.5,-19.5)|{x_4}}{\color{red}\ar@{->}"1"+(-7.5,-15.5);"1"+(-1,-20)|{x_6}}\ar@{->}"1"+(-1,-22);"1"+(-7.5,-33.5)|{x_1} \ar@{->}"1"+(-8.8,-33);"1"+(-8.8,-17)|{x_2} \ar@{->}"1"+(1,-20);"1"+(7.5,-15.5)|{x_7} \ar@{.>}"1"+(8.8,-17);"1"+(8.8,-33)|{x_2} {\color{red}\ar@{->}"1"+(7.5,-33.5);"1"+(1,-22)|{x_3}}\ar@{.>}"1"+(8.8,-13);"1"+(8.8,-6.5)|{x_8} \ar@{.>}"1"+(16.5,-19.5);"1"+(10,-15.5)|{x_9}
  \end{xy}
 \hspace{10mm}
  \begin{xy}
   (0,0)="0", +(17.5,10)="1", +(17.5,-10)="2", +(0,-20)="3", +(-17.5,-10)="4", +(-17.5,10)="5", "5"+(2.5,0.5)*{{\color{red}\bullet}}, "4"+(2,3)*{{\color{red}\bullet}}, "1"+(-4,-4)*{{\color{red}\bullet}}
   \ar@{-}"0";"1"|{4} \ar@{.}"1";"2"|{9} \ar@{.}"2";"3"|{8} \ar@{-}"3";"4"|{7} \ar@{-}"4";"5"|{6} \ar@{-}"5";"0"|{5} \ar@{-}"1";"5"|{1} \ar@{-}"1";"4"|{2} \ar@{-}"1";"3"|{3} {\color{red}\ar@{->}"1"+(-8.8,-6.5);"1"+(-8.8,-13)|{x_5}}\ar@{->}"1"+(-10,-15.5);"1"+(-16.5,-19.5)|{x_4} \ar@{->}"1"+(-7.5,-15.5);"1"+(-1,-20)|{x_6} \ar@{->}"1"+(-1,-22);"1"+(-7.5,-33.5)|{x_1} {\color{red}\ar@{->}"1"+(-8.8,-33);"1"+(-8.8,-17)|{x_2}}\ar@{->}"1"+(1,-20);"1"+(7.5,-15.5)|{x_7} \ar@{.>}"1"+(8.8,-17);"1"+(8.8,-33)|{x_2} {\color{red}\ar@{->}"1"+(7.5,-33.5);"1"+(1,-22)|{x_3}}\ar@{.>}"1"+(8.8,-13);"1"+(8.8,-6.5)|{x_8} \ar@{.>}"1"+(16.5,-19.5);"1"+(10,-15.5)|{x_9}
  \end{xy}
\]
Thus we have $f^{[1,2]}=x_1x_4x_7+x_3x_4x_6+x_2x_3x_5$.

 Finally, we give an interpretation of maximal discrete subsets of $\overline{Q}$ in terms of a quiver with potential $(\widetilde{Q},\widetilde{W})$ given by Demonet and Luo \cite{DL}. We introduce the notion of {\it minimal cuts} of quivers with potential (Definition \ref{mincut}), and prove that they correspond bijectively with maximal discrete subsets of $\overline{Q}$. In particular, our cluster expansion formula can be written in terms of minimal cuts (Corollary \ref{cutformula}).



\section{Preliminaries}

\subsection{Cluster algebras with coefficients from ice quivers}

 We begin with recalling the definition of cluster algebras with coefficients associated with ice quivers \cite{K}. In the rest, we fix positive integers $n \le m$. An {\it ice quiver of type $(n,m)$} is a quiver $Q$ with vertices $Q_0=[1,m]$ such that there are no arrows between vertices in $[n+1,m]$. The elements of $[n+1,m]$ are called {\it frozen vertices}.

 To define cluster algebras with coefficients from ice quivers, we need to prepare some notations. Let ${\mathcal F}:={\mathbb Q}(t_1,\ldots,t_m)$ be a field of rational functions in $m$ variables over ${\mathbb Q}$.

\begin{definition}
 A {\it seed} is a pair $(x,Q)$ consisting of the following data:\par
 (i) $Q$ is an ice quiver of type $(n,m)$ without loops and $2$-cycles.\par
 (ii) $x=(x_1,\ldots,x_m)$ is a free generating set of ${\mathcal F}$ over ${\mathbb Q}$. Then each $x_i$ is called a {\it cluster variable} for $i \in [1,n]$ and a {\it coefficient} for $i \in [n+1,m]$.
\end{definition}

\begin{definition}
 For a seed $(x,Q)$, the $\it{mutation}$ $\mu_k(x,Q)=(x',Q')$ in direction $k$ $(1 \le k \le n)$ is defined as follows.\par
 (i) $Q'$ is the ice quiver obtained from $Q$ by the following steps:\par
 \hspace{5mm}(1) For any path $i \rightarrow k \rightarrow j$, add an arrow $i \rightarrow j$.\par
 \hspace{5mm}(2) Reverse all arrows incident to $k$.\par
 \hspace{5mm}(3) Remove a maximal set of disjoint $2$-cycles.\par
 \hspace{5mm}(4) Remove all arrow connecting two frozen vertices.\par
 (ii) $x'=(x'_1,\ldots,x'_m)$ is defined by
 \begin{equation*}
  x_k x'_k = \prod_{(j \rightarrow k) \in Q_1}x_j+\prod_{(j \leftarrow k) \in Q_1}x_j \ \ \mbox{and} \ \ x'_i = x_i \ \ \mbox{if} \ \ i \neq k.
 \end{equation*}
\end{definition}

Then it is elementary that $\mu_k(x,Q)=(x',Q')$ is also a seed. Moreover, $\mu_k$ is an involution, that is, we have $\mu_k\mu_k(x,Q)=(x,Q)$.\par

Now we define cluster algebras with coefficients. For an ice quiver $Q$, we fix a seed $(x=(x_1,\ldots,x_m),Q)$ which we call an {\it initial seed}. We also call each $x_i$ an {\it initial cluster variable}.

\begin{definition}\label{clusteralg}
The {\it cluster algebra} ${\mathcal A}(Q)={\mathcal A}(x,Q)$ {\it with coefficients} for the initial seed $(x,Q)$ is a $\mathbb{Z}$-subalgebra of ${\mathcal F}$ generated by the cluster variables and the coefficients obtained by all sequences of mutations from $(x,Q)$.
\end{definition}

\begin{example}\label{ex1}
The quiver $Q=\Biggl[$
\begin{xy}
 (0,5)="1"*{1}, +(0,-8)="2"*{2}, "1"+(-7,0)="3"*{3}, "1"+(7,0)="4"*{4}, "2"+(-7,0)="5"*{5}, "2"+(7,0)="6"*{6}
 \ar@{>}"1"+(0,-2);"2"+(0,2) \ar@{>}"3"+(1.5,0);"1"+(-1.5,0) \ar@{>}"1"+(1.5,0);"4"+(-1.5,0) \ar@{>}"2"+(-1.5,0);"5"+(1.5,0) \ar@{>}"6"+(-1.5,0);"2"+(1.5,0)
\end{xy}
$\Biggr]$ is an ice quiver of type $(2,6)$ with frozen vertices $[3,6]$. Let $(x=(x_1,x_2,x_3,x_4,x_5,x_6),Q)$ be a seed and we consider mutations of $(x,Q)$.

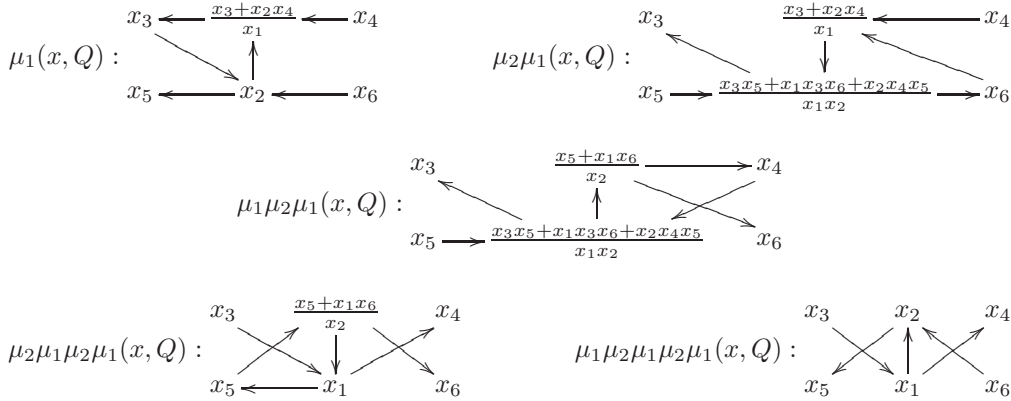
\[
  \mu_1(x,Q) :
  \begin{xy}
   (0,5)="1"*{\frac{x_3+x_2x_4}{x_1}}, +(0,-10)="2"*{x_2}, "1"+(-15,0)="3"*{x_3}, "1"+(15,0)="4"*{x_4}, "2"+(-15,0)="5"*{x_5}, "2"+(15,0)="6"*{x_6}
   \ar@{>}"2"+(0,2);"1"+(0,-3) \ar@{>}"1"+(-6.5,0);"3"+(2.5,0) \ar@{>}"4"+(-2.5,0);"1"+(6.5,0) \ar@{>}"2"+(-2.5,0);"5"+(2.5,0) \ar@{>}"6"+(-2.5,0);"2"+(2.5,0) \ar@{>}"3"+(2,-2);"2"+(-2,2)
  \end{xy}
 \hspace{15mm} \mu_2\mu_1(x,Q) :
  \begin{xy}
   (0,5)="1"*{\frac{x_3+x_2x_4}{x_1}}, +(0,-10)="2"*{\frac{x_3x_5+x_1x_3x_6+x_2x_4x_5}{x_1x_2}}, "1"+(-23,0)="3"*{x_3}, "1"+(23,0)="4"*{x_4}, "2"+(-23,0)="5"*{x_5}, "2"+(23,0)="6"*{x_6}
   \ar@{>}"1"+(0,-3);"2"+(0,3) \ar@{>}"4"+(-2.5,0);"1"+(6.5,0) \ar@{>}"5"+(2.5,0);"2"+(-15,0) \ar@{>}"2"+(15,0);"6"+(-2.5,0) \ar@{>}"2"+(-10,3);"3"+(2,-2) \ar@{>}"6"+(-2,2);"1"+(5,-2)
  \end{xy}
\]
\vspace{3mm}
\[
  \mu_1\mu_2\mu_1(x,Q) :
  \begin{xy}
   (0,5)="1"*{\frac{x_5+x_1x_6}{x_2}}, +(0,-10)="2"*{\frac{x_3x_5+x_1x_3x_6+x_2x_4x_5}{x_1x_2}}, "1"+(-23,0)="3"*{x_3}, "1"+(23,0)="4"*{x_4}, "2"+(-23,0)="5"*{x_5}, "2"+(23,0)="6"*{x_6}
   \ar@{>}"2"+(0,3);"1"+(0,-3) \ar@{>}"1"+(6.5,0);"4"+(-2.5,0) \ar@{>}"5"+(2.5,0);"2"+(-15,0) \ar@{>}"4"+(-2,-2);"2"+(10,3) \ar@{>}"2"+(-10,3);"3"+(2,-2) \ar@{>}"1"+(5,-2);"6"+(-2,2)
  \end{xy}
\]
\vspace{3mm}
\[
  \mu_2\mu_1\mu_2\mu_1(x,Q) :
  \begin{xy}
   (0,5)="1"*{\frac{x_5+x_1x_6}{x_2}}, +(0,-10)="2"*{x_1}, "1"+(-15,0)="3"*{x_3}, "1"+(15,0)="4"*{x_4}, "2"+(-15,0)="5"*{x_5}, "2"+(15,0)="6"*{x_6}
   \ar@{>}"1"+(0,-3);"2"+(0,2) \ar@{>}"2"+(-2.5,0);"5"+(2.5,0) \ar@{>}"3"+(2,-2);"2"+(-2,2) \ar@{>}"2"+(2,2);"4"+(-2,-2) \ar@{>}"5"+(2,2);"1"+(-5,-2) \ar@{>}"1"+(5,-2);"6"+(-2,2)
  \end{xy}
 \hspace{15mm} \mu_1\mu_2\mu_1\mu_2\mu_1(x,Q) :
  \begin{xy}
   (0,5)="1"*{x_2}, +(0,-10)="2"*{x_1}, "1"+(-12,0)="3"*{x_3}, "1"+(12,0)="4"*{x_4}, "2"+(-12,0)="5"*{x_5}, "2"+(12,0)="6"*{x_6}
   \ar@{>}"2"+(0,2);"1"+(0,-2) \ar@{>}"3"+(2,-2);"2"+(-2,2) \ar@{>}"2"+(2,2);"4"+(-2,-2) \ar@{>}"1"+(-2,-2);"5"+(2,2) \ar@{>}"6"+(-2,2);"1"+(2,-2)
  \end{xy}
\]\vspace{2mm}

 The calculation is periodic, hence we got all the cluster variables. Therefore, the cluster algebra is
 \begin{equation*}
  {\mathcal A}(Q)=\mathbb{Z}\bigl[x_1,x_2,\mbox{$\frac{x_3+x_2x_4}{x_1},\frac{x_5+x_1x_6}{x_2},\frac{x_3x_5+x_1x_3x_6+x_2x_4x_5}{x_1x_2}$},x_3,x_4,x_5,x_6\bigr].
 \end{equation*}
\end{example}
\vspace{1mm}


\subsection{The cluster expansion formula of Musiker and Schiffler}\label{subsectionMS}

 Musiker and Schiffler \cite{MS} gave a cluster expansion formula using a {\it labeled graph} and its {\it perfect matchings}. Let $T$ be a triangulated $(n+3)$-gon such that $Q_T$ is acyclic. We recall the construction of Musiker and Schiffler. For more details, see \cite{MS}. The construction can be followed step by step in Example \ref{MSex}. As $Q_T$ is acyclic, at most two sides of each triangle of $T$ are diagonals. The labeled graph $\overline{G}:=\overline{G}_{T}$ is obtained from $T$ by {\it unfolding} along the third side of each of triangles of $T$, two sides of which are diagonals. We label all edges of $\overline{G}$ by the corresponding arcs of $T$.
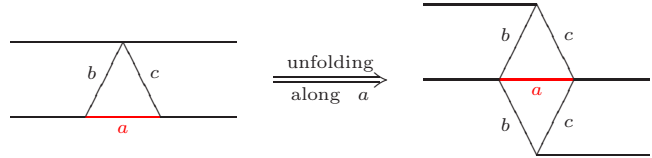
\[
  \begin{xy}
   (0,0)="0", +(15,0)="1", +(15,0)="2", "0"+(0,-10)="3", +(10,0)="4", +(10,0)="5", +(10,0)="6", "2"+(5,-5)="A", +(15,0)="B", +(5,10)="00", +(15,0)="01", "00"+(0,-10)="02", +(10,0)="03", +(10,0)="04", +(10,0)="05", "03"+(5,-10)="06", +(15,0)="07"
   \ar@{-}"0";"1" \ar@{-}"1";"2" \ar@{-}"3";"4" {\color{red}\ar@{-}"4";"5"_{a}}\ar@{-}"5";"6" \ar@{-}"1";"4"_{b} \ar@{-}"1";"5"^{c} \ar@{=>}"A";"B"^{{\rm unfolding}}_{{\rm along}\ \ a} \ar@{-}"00";"01" \ar@{-}"02";"03" {\color{red}\ar@{-}"03";"04"_{a}}\ar@{-}"04";"05" \ar@{-}"06";"07" \ar@{-}"01";"03"_{b} \ar@{-}"01";"04"^{c} \ar@{-}"03";"06"_{b} \ar@{-}"04";"06"^{c}
  \end{xy}
\]

 Let {\large $\boxbslash_k$} be the square of $\overline{G}$ with diagonal $k$, and $\overline{G}^{[i,j]}$ the minimal subgraph of $\overline{G}$ containing {\large $\boxbslash_i$}, {\large $\boxbslash_{i+1}$}, $\ldots$, {\large $\boxbslash_{j}$} for $1 \le i \le j \le n$. Let $G := G_{T}$ be the graph obtained from $\overline{G}$ by removing the diagonal in each {\large $\boxbslash_k$}. Similarly, let $\square_k$ and $G^{[i,j]}$ be the graphs obtained from {\large $\boxbslash_k$} and $\overline{G}^{[i,j]}$ in the same way as $G$.

\begin{definition}\label{pmside}
 A {\it perfect matching} in $G^{[i,j]}$ is a subset $P$ of the edges in $G^{[i,j]}$ such that each vertex is contained in exactly one edge on $P$. We denote by $\mathbb{P}(G^{[i,j]})$ the set of all perfect matchings in $G^{[i,j]}$.
\end{definition}

 The formula of Musiker and Schiffler is the following cluster expansion formula obtained by the perfect matchings in $G^{[i,j]}$. 

\begin{theorem}\label{MS}\cite{MS}
 For $1 \le i \le j \le n$, we have
 \begin{equation*}
  f^{[i,j]}=\sum_{P \in \mathbb{P}(G^{[i,j]})}\prod_{e \in P}x_{e}.
 \end{equation*}
\end{theorem}

\begin{example}\label{MSex}
 For the quiver $Q=[ 1 \rightarrow 2 \rightarrow 3 ]$, we have
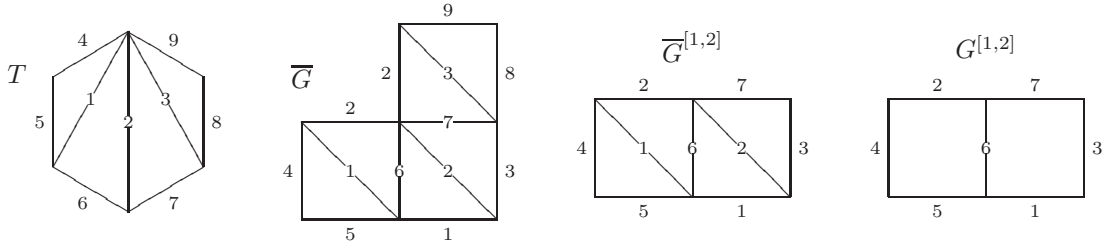
\[
  T
  \begin{xy}
   (0,0)="0", +(10,6)="1", +(10,-6)="2", +(0,-12)="3", +(-10,-6)="4", +(-10,6)="5", 
   \ar@{-}"0";"1"^{4} \ar@{-}"1";"2"^{9} \ar@{-}"2";"3"^{8} \ar@{-}"3";"4"^{7} \ar@{-}"4";"5"^{6} \ar@{-}"5";"0"^{5} \ar@{-}"1";"5"|{1} \ar@{-}"1";"4"|{2} \ar@{-}"1";"3"|{3}
  \end{xy}\hspace{6mm}
  \begin{xy}
   (0,-6)="1", +(13,0)="2", +(0,-13)="3", +(-13,0)="4", "2"+(13,0)="5", +(0,-13)="6", "2"+(0,13)="7", +(13,0)="8", "1"+(0,6)*{\overline{G}}
   \ar@{-}"1";"2"^{2} \ar@{-}"2";"3"|{6} \ar@{-}"3";"4"^{5} \ar@{-}"4";"1"^{4} \ar@{-}"2";"5"|{7} \ar@{-}"5";"6"^{3} \ar@{-}"6";"3"^{1} \ar@{-}"1";"3"|{1} \ar@{-}"2";"6"|{2} \ar@{-}"2";"7"^{2} \ar@{-}"7";"8"^{9} \ar@{-}"8";"5"^{8} \ar@{-}"7";"5"|{3}
  \end{xy}\hspace{6mm}
  \begin{xy}
   (0,4)*{\overline{G}^{[1,2]}}, +(-13,-7)="01", +(13,0)="02", +(0,-13)="03", +(-13,0)="04", "02"+(13,0)="05", +(0,-13)="06"
   \ar@{-}"01";"02"^{2} \ar@{-}"02";"03"|{6} \ar@{-}"03";"04"^{5} \ar@{-}"04";"01"^{4} \ar@{-}"02";"05"^{7} \ar@{-}"05";"06"^{3} \ar@{-}"06";"03"^{1} \ar@{-}"01";"03"|{1} \ar@{-}"02";"06"|{2}
  \end{xy}\hspace{6mm}
  \begin{xy}
   (0,4)*{G^{[1,2]}}, +(-13,-7)="01", +(13,0)="02", +(0,-13)="03", +(-13,0)="04", "02"+(13,0)="05", +(0,-13)="06"
   \ar@{-}"01";"02"^{2} \ar@{-}"02";"03"|{6} \ar@{-}"03";"04"^{5} \ar@{-}"04";"01"^{4} \ar@{-}"02";"05"^{7} \ar@{-}"05";"06"^{3} \ar@{-}"06";"03"^{1}
  \end{xy}
\]
Then $G^{[1,2]}$ has the following three perfect matchings:

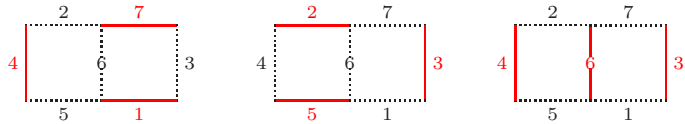
\[
  \begin{xy}
   (0,0)="01", +(10,0)="02", +(0,-10)="03", +(-10,0)="04", "02"+(10,0)="05", +(0,-10)="06"
   \ar@{.}"01";"02"^{2} \ar@{.}"02";"03"|{6} \ar@{.}"03";"04"^{5} {\color{red}\ar@{-}"04";"01"^{4}}{\color{red}\ar@{-}"02";"05"^{7}}\ar@{.}"05";"06"^{3} {\color{red}\ar@{-}"06";"03"^{1}}
  \end{xy}\hspace{6mm}
  \begin{xy}
   (0,0)="01", +(10,0)="02", +(0,-10)="03", +(-10,0)="04", "02"+(10,0)="05", +(0,-10)="06"
   {\color{red}\ar@{-}"01";"02"^{2}}\ar@{.}"02";"03"|{6} {\color{red}\ar@{-}"03";"04"^{5}}\ar@{.}"04";"01"^{4} \ar@{.}"02";"05"^{7} {\color{red}\ar@{-}"05";"06"^{3}}\ar@{.}"06";"03"^{1}
  \end{xy}\hspace{12mm}
  \begin{xy}
   (0,0)="01", +(10,0)="02", +(0,-10)="03", +(-10,0)="04", "02"+(10,0)="05", +(0,-10)="06"
   \ar@{.}"01";"02"^{2} {\color{red}\ar@{-}"02";"03"|{6}}\ar@{.}"03";"04"^{5} {\color{red}\ar@{-}"04";"01"^{4}}\ar@{.}"02";"05"^{7} {\color{red}\ar@{-}"05";"06"^{3}}\ar@{.}"06";"03"^{1}
  \end{xy}
\]
Thus we have $f^{[1,2]}=x_1x_4x_7+x_2x_3x_5+x_3x_4x_6$.
\end{example}



\section{Proofs of Theorem \ref{angle} and Corollary \ref{main1}}

 We will show Theorem \ref{angle} and Corollary \ref{main1} only for $[i,j]=[1,n]$ since $f^{[i,j]}$ in ${\mathcal A}(\overline{Q})$ coincides with $f^{[i,j]}$ in ${\mathcal A}(\overline{Q}^{[i,j]})$ by Theorem \ref{MS}. Let $T$ be a triangulated $(n+3)$-gon such that $Q:=Q_T$ is an acyclic quiver, that is of type $A_n$.



\subsection{Proof of Theorem \ref{angle}}

 We prepare some notations to consider perfect matchings of angles in the triangulated $(n+3)$-gon $T$.\\
 (a) As $Q$ is of type $A_n$, it is possible, in a unique way, to label the triangles appearing in $T$ as $\triangle_0, \triangle_1, \ldots, \triangle_n$, such that two sides of $\triangle_i$ (resp., one side of $\triangle_0$, $\triangle_n$) are the diagonals $i$ and $i+1$ (resp., the diagonal $1$, the diagonal $n$) for $i \in [1,n-1]$.\\
 (b) Label $n+1$ vertices of $T$ by $v_0, \ldots, v_n$ such that $v_0$ is incident to the diagonal $1$ and not incident to the diagonal $2$, and $v_k$ is incident to the diagonal $k$ for $k \in [1,n]$.\\
 (c) Label by $a_{ij}$ the angle of $\triangle_i$ at the vertex $v_j$. Let $A(T)$ be the set of all angles in $T$ that are labeled in this way.
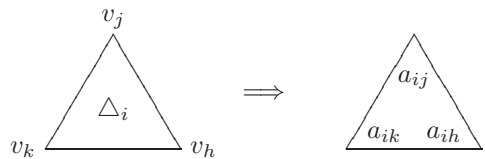
\[
  \begin{xy}
   (0,0)="0", +(9,-15.3)="1", +(-18,0)="2", "0"+(20,-8)*{\Longrightarrow}, +(20,8)="00", +(9,-15.3)="01", +(-18,0)="02", "0"+(0,2)*{v_j}, "1"+(3,0)*{v_h}, "2"+(-3,0)*{v_k}, "0"+(0,-10)*{\triangle_i}, "00"+(0,-6)*{a_{ij}}, "01"+(-5,2)*{a_{ih}}, "02"+(5,2)*{a_{ik}}
   \ar@{-}"0";"1" \ar@{-}"1";"2" \ar@{-}"2";"0" \ar@{-}"00";"01" \ar@{-}"01";"02" \ar@{-}"02";"00"
  \end{xy}
\]

 For example, the labellings of $T$ for the case $Q=[ 1 \rightarrow 2 \rightarrow 3 \leftarrow 4 \leftarrow 5 ]$ are the following:
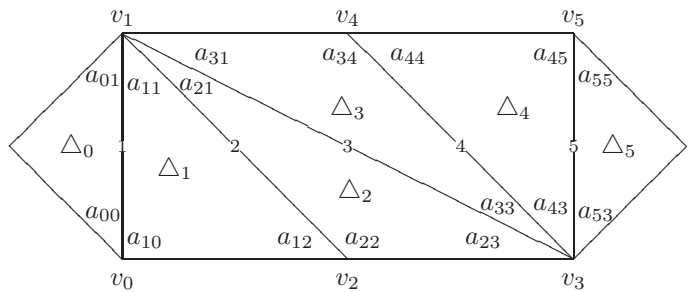
\[
  \begin{xy}
   (0,0)="00", +(15,15)="01", +(30,0)="02", +(30,0)="03", +(15,-15)="04", +(-15,-15)="05", +(-30,0)="06", +(-30,0)="07", "01"+(-2.5,-6)*{a_{01}}, "01"+(3,-7)*{a_{11}}, "01"+(10,-7)*{a_{21}}, "01"+(12,-3)*{a_{31}}, "02"+(-1,-3)*{a_{34}}, "02"+(8,-3)*{a_{44}}, "03"+(-3,-3)*{a_{45}}, "03"+(3,-6)*{a_{55}}, "05"+(3,6)*{a_{53}}, "05"+(-3,7)*{a_{43}}, "05"+(-10,7)*{a_{33}}, "05"+(-12,2.5)*{a_{23}}, "06"+(2,2.5)*{a_{22}}, "06"+(-7,2.5)*{a_{12}}, "07"+(3,2.5)*{a_{10}}, "07"+(-2.5,6)*{a_{00}}, "00"+(9,0)*{\triangle_0}, "07"+(7,12)*{\triangle_1}, "06"+(1,9)*{\triangle_2}, "01"+(30,-10)*{\triangle_3}, "02"+(22,-10)*{\triangle_4}, "04"+(-9,0)*{\triangle_5}, "01"+(0,2)*{v_1}, "07"+(0,-3)*{v_0}, "06"+(0,-3)*{v_2}, "05"+(0,-3)*{v_3}, "02"+(0,2)*{v_4}, "03"+(0,2)*{v_5}
   \ar@{-}"00";"01" \ar@{-}"01";"02" \ar@{-}"02";"03" \ar@{-}"03";"04" \ar@{-}"04";"05" \ar@{-}"05";"06" \ar@{-}"06";"07" \ar@{-}"07";"00" \ar@{-}"01";"07"|{1} \ar@{-}"01";"06"|{2} \ar@{-}"01";"05"|{3} \ar@{-}"02";"05"|{4} \ar@{-}"03";"05"|{5}
  \end{xy}
\]
 Thus $A(T)=\{a_{00},a_{01},a_{10},a_{11},a_{12},a_{20},a_{22},a_{23},a_{30},a_{33},a_{34},a_{43},a_{44},a_{45},a_{53},a_{55}\}.$

 Under the above parametrization of the angles in $T$, the following is clear.

\begin{lemma}\label{A1A2}
 A subset $A$ of $A(T)$ is a perfect matching in the sense of Definition \ref{perangle} if and only if it satisfies the following conditions.\par
(A1) For any $i \in [0,n]$, there is a unique $j \in [0,n]$ such that $a_{ij} \in A$,\par
(A2) For any $j \in [0,n]$, there is a unique $i \in [0,n]$ such that $a_{ij} \in A$.\\
\end{lemma}


 We denote by $G_1$ the set of edges of $G$. Let $A(\overline{G})$ be the set of angles between a diagonal and a side of a {\large $\boxbslash_k$} in $\overline{G}$, and $\overline{\varphi} : A(\overline{G}) \rightarrow G_1$ the map sending $a \in A(\overline{G})$ to the side that is opposite to $a$. Clearly, $\overline{\varphi} : A(\overline{G}) \rightarrow G_1$ is surjective. By definition of the unfolding process (see Subsection \ref{subsectionMS}), there is a canonical surjection $\pi : A(\overline{G}) \rightarrow A(T)$ compatible with the construction of $\overline{G}$.

\begin{lemma}\label{lembij}
 There exists a bijection $\varphi : A(T) \rightarrow G_1$ making the following diagram commutative:
\[ \xymatrix@C=0.3mm@R=7mm{
 & A(\overline{G}) \ar@{->>}[ld]_{\pi} \ar@{->}[rd]^{\overline{\varphi}} & \\
 A(T) \ar@{->}[rr]^{\sim}_{\varphi} & & G_1
} \]
\end{lemma}

 Before proving Lemma \ref{lembij}, we give simple examples.

\begin{example}\label{exmaps}
 ($n=1$) For the case $Q=[ 1 ]$, we have a natural identification of $\overline{G}$ with $T$.
\\
 ($n=2$) For the case $Q=[ 1 \rightarrow 2 ]$, the following diagonals show that there is a bijection $\varphi : A(T) \rightarrow G_1$.
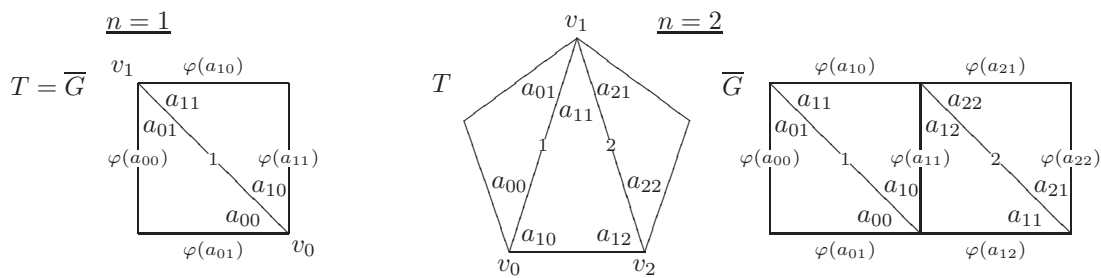
\[
  \begin{xy}
   (0,0)="0", +(20,0)="1", +(0,-20)="2", +(-20,0)="3", "0"+(3,-6)*{a_{01}}, "0"+(6,-2.5)*{a_{11}}, "2"+(-2.5,6)*{a_{10}}, "2"+(-6,2)*{a_{00}}, "0"+(0,8)*{\underline{n=1}}, "0"+(-12,0)*{T = \overline{G}}, "0"+(-2,2)*{v_{1}}, "2"+(2,-2)*{v_0}
   \ar@{-}"0";"1"^{\varphi(a_{10})} \ar@{-}"1";"2"|{\varphi(a_{11})} \ar@{-}"2";"3"^{\varphi(a_{01})} \ar@{-}"3";"0"|{\varphi(a_{00})} \ar@{-}"0";"2"|{1}
  \end{xy}\hspace{15mm}
  \begin{xy}
   (0,-5)="0", +(15,11)="1", +(15,-11)="2", +(-6,-17.5)="3", +(-18,0)="4", "1"+(-5,-7)*{a_{01}}, "1"+(0,-10)*{a_{11}}, "1"+(5,-7)*{a_{21}}, "4"+(0,9)*{a_{00}}, "4"+(4,2)*{a_{10}}, "3"+(-4,2)*{a_{12}}, "3"+(0,9)*{a_{22}}, "1"+(0,2)*{v_1}, "4"+(0,-2)*{v_0}, "3"+(0,-2)*{v_2}, "0"+(-3,5)*{T}, "2"+(0,13)*{\underline{n=2}}
   \ar@{-}"0";"1" \ar@{-}"1";"2" \ar@{-}"2";"3" \ar@{-}"3";"4" \ar@{-}"4";"0" \ar@{-}"1";"4"|{1} \ar@{-}"1";"3"|{2}
  \end{xy}
  \begin{xy}
   (0,0)="0", +(20,0)="1", +(0,-20)="2", +(-20,0)="3", "1"+(20,0)="4", +(0,-20)="5", "0"+(3,-6)*{a_{01}}, "2"+(-2.5,6)*{a_{10}}, "1"+(6,-2.5)*{a_{22}}, "0"+(6,-2.5)*{a_{11}}, "2"+(-6,2)*{a_{00}}, "1"+(3,-6)*{a_{12}}, "5"+(-2.5,6)*{a_{21}}, "5"+(-6,2)*{a_{11}}, "0"+(-5,0)*{\overline{G}}
   \ar@{-}"0";"1"^{\varphi(a_{10})} \ar@{-}"1";"2"|{\varphi(a_{11})} \ar@{-}"2";"3"^{\varphi(a_{01})} \ar@{-}"3";"0"|{\varphi(a_{00})} \ar@{-}"1";"4"^{\varphi(a_{21})} \ar@{-}"4";"5"|{\varphi(a_{22})} \ar@{-}"5";"2"^{\varphi(a_{12})} \ar@{-}"0";"2"|{1} \ar@{-}"1";"5"|{2}
  \end{xy}
\]
 ($n=3$) For the case $Q=[1 \rightarrow 2 \rightarrow 3]$, the following diagonals show that there is a bijection $\varphi : A(T) \rightarrow G_1$.
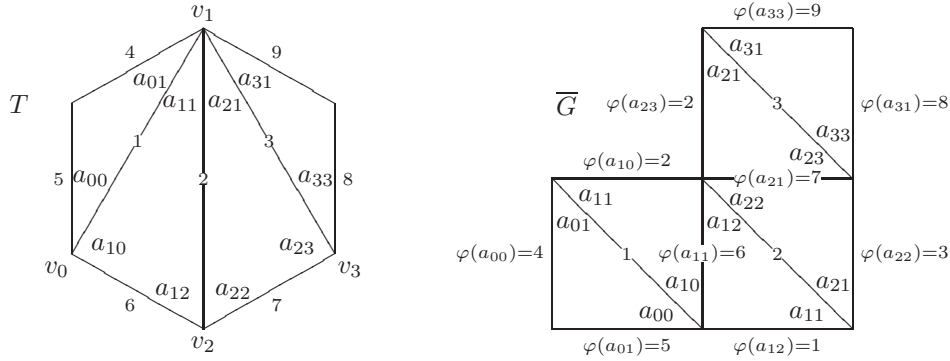
\[
  \begin{xy}
   (0,0)="0", +(17.5,10)="1", +(17.5,-10)="2", +(0,-20)="3", +(-17.5,-10)="4", +(-17.5,10)="5", "0"+(-7,0)*{T}, "2"+(10,-10)*{}, "1"+(-7,-7)*{a_{01}}, "1"+(-3,-10)*{a_{11}}, "1"+(3,-10)*{a_{21}}, "1"+(7,-7)*{a_{31}}, "3"+(-2.5,10)*{a_{33}}, "3"+(-5,1)*{a_{23}}, "4"+(4,5)*{a_{22}}, "4"+(-4,5)*{a_{12}}, "5"+(5,1)*{a_{10}}, "5"+(2.5,10)*{a_{00}}, "1"+(0,2)*{v_1}, "5"+(-2,-2)*{v_0}, "4"+(0,-2)*{v_2}, "3"+(2,-2)*{v_3}
   \ar@{-}"0";"1"^{4} \ar@{-}"1";"2"^{9} \ar@{-}"2";"3"^{8} \ar@{-}"3";"4"^{7} \ar@{-}"4";"5"^{6} \ar@{-}"5";"0"^{5} \ar@{-}"1";"5"|{1} \ar@{-}"1";"4"|{2} \ar@{-}"1";"3"|{3} 
  \end{xy}\hspace{5mm}
  \begin{xy}
   (0,-10)="0", +(20,0)="1", +(0,-20)="2", +(-20,0)="3", "1"+(20,0)="4", +(0,-20)="5", "1"+(0,20)="6", +(20,0)="7", "0"+(3,-6)*{a_{01}}, "0"+(6,-2.5)*{a_{11}}, "2"+(-2.5,6)*{a_{10}}, "2"+(-6,2)*{a_{00}}, "1"+(3,-6)*{a_{12}}, "1"+(6,-3)*{a_{22}}, "5"+(-2.5,6)*{a_{21}}, "5"+(-6,2)*{a_{11}}, "4"+(-6,3)*{a_{23}}, "4"+(-2.5,6)*{a_{33}}, "6"+(6,-2.5)*{a_{31}}, "6"+(3,-6)*{a_{21}}, "0"+(2,10)*{\overline{G}}
   \ar@{-}"0";"1"^{\varphi(a_{10})=2} \ar@{-}"1";"2"|{\varphi(a_{11})=6} \ar@{-}"2";"3"^{\varphi(a_{01})=5} \ar@{-}"3";"0"^{\varphi(a_{00})=4} \ar@{-}"1";"4"|{\varphi(a_{21})=7} \ar@{-}"4";"5"^{\varphi(a_{22})=3} \ar@{-}"5";"2"^{\varphi(a_{12})=1} \ar@{-}"1";"6"^{\varphi(a_{23})=2} \ar@{-}"6";"7"^{\varphi(a_{33})=9} \ar@{-}"7";"4"^{\varphi(a_{31})=8} \ar@{-}"0";"2"|{1} \ar@{-}"1";"5"|{2} \ar@{-}"6";"4"|{3}
  \end{xy}
\]
\end{example}

\begin{proof}[Proof of Lemma \ref{lembij}]
 We construct $\varphi$ by induction on $n$. For $n=1$, we have a natural identification of $\overline{G}$ with $T$ which gives the desired map. Assume that we constructed a bijection $\varphi : A(T^{[1,n-1]}) \rightarrow G^{[1,n-1]}_1$ making the following diagram commutative:
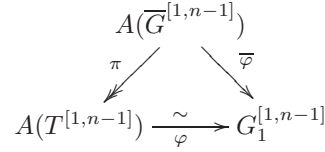
\[ \xymatrix@!=5mm{
 & A(\overline{G}^{[1,n-1]}) \ar@{->>}[ld]_{\pi} \ar@{->}[rd]^{\overline{\varphi}} & \\
 A(T^{[1,n-1]}) \ar@{->}[rr]^{\sim}_{\varphi} & & G^{[1,n-1]}_1
} \]
 Then $A(T)$ and $G_1$ have $3$ additional elements colored in red:
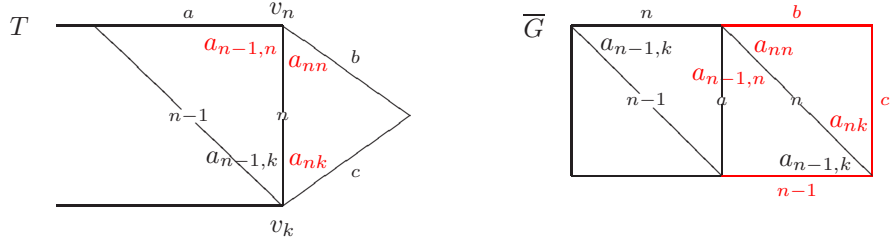
\[
  \begin{xy}
   (0,0)="0", +(5,0)="1", +(25,0)="2", +(17,-12)="3", +(-17,-12)="4", +(-30,0)="5", "2"+(0,2)*{v_n}, "4"+(0,-3)*{v_k}, "2"+(-5.5,-3)*{{\color{red} a_{n-1,n}}}, "2"+(3.5,-5)*{{\color{red} a_{nn}}}, "4"+(3.5,6)*{{\color{red} a_{nk}}}, "4"+(-5.3,6)*{a_{n-1,k}}, "0"+(-5,0)*{T}
  \ar@{-}"0";"1" \ar@{-}"1";"2"^{a} \ar@{-}"2";"3"^{b} \ar@{-}"3";"4"^{c} \ar@{-}"4";"5" \ar@{-}"1";"4"|{n-1} \ar@{-}"2";"4"|{n}
  \end{xy}
 \hspace{15mm}
  \begin{xy}
   (0,0)="0", +(20,0)="1", +(0,-20)="2", +(-20,0)="3", "1"+(20,0)="4", +(0,-20)="5", "1"+(1,-7)*{{\color{red} a_{n-1,n}}}, "1"+(7,-3)*{{\color{red} a_{nn}}}, "5"+(-3,7)*{{\color{red} a_{nk}}}, "5"+(-7.5,1.5)*{a_{n-1,k}}, "0"+(8.7,-3)*{a_{n-1,k}}, "0"+(-5,0)*{\overline{G}}
   \ar@{-}"0";"1"^{n} \ar@{-}"1";"2"|{a} \ar@{-}"2";"3" \ar@{-}"3";"0" {\color{red} \ar@{-}"1";"4"^{b}}{\color{red} \ar@{-}"4";"5"^{c}}{\color{red} \ar@{-}"5";"2"^{n-1}}\ar@{-}"0";"2"|{n-1} \ar@{-}"1";"5"|{n}
  \end{xy}
\]
 We extend the map $\varphi : A(T^{[1,n-1]}) \rightarrow G^{[1,n-1]}_1$ to $\varphi : A(T) \rightarrow G_1$ as in the following diagram:
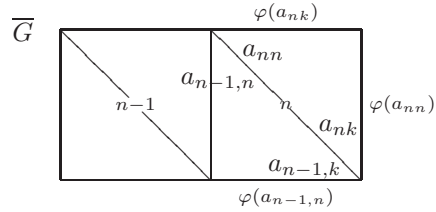
\[
\begin{xy}
   (0,22)="0", +(20,0)="1", +(0,-20)="2", +(-20,0)="3", "1"+(20,0)="4", +(0,-20)="5", "1"+(1,-7)*{a_{n-1,n}}, "1"+(7,-3)*{a_{nn}}, "5"+(-3,7)*{a_{nk}}, "5"+(-7.5,1.5)*{a_{n-1,k}}, "0"+(-5,0)*{\overline{G}}
   \ar@{-}"0";"1" \ar@{-}"1";"2" \ar@{-}"2";"3" \ar@{-}"3";"0" \ar@{-}"1";"4"^{\varphi(a_{nk})} \ar@{-}"4";"5"^{\varphi(a_{nn})} \ar@{-}"5";"2"^{\varphi(a_{n-1,n})} \ar@{-}"0";"2"|{n-1} \ar@{-}"1";"5"|{n}
  \end{xy} \qedhere
\]
\end{proof}


 The following is a key proposition to show Theorem \ref{angle}.

\begin{proposition}\label{PMPM}
 The map $\varphi$ of Lemma \ref{lembij} induces a bijection $\varphi : {\mathbb A}(T) \rightarrow {\mathbb P}(G)$.
\end{proposition}


 To prove Proposition \ref{PMPM}, we consider the set of all perfect matchings of $G$.

 (1) If $n=1$, the following is a complete list of ${\mathbb P}(G)$:
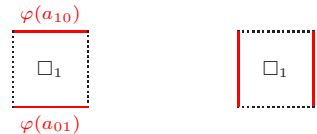
\[
  \begin{xy}
   (0,0)="0", +(10,0)="1", +(0,-10)="2", +(-10,0)="3"
   {\color{red}\ar@{-}"0";"1"^{\varphi(a_{10})}}\ar@{.}"1";"2" {\color{red}\ar@{-}"2";"3"^{\varphi(a_{01})}}\ar@{.}"3";"0" \ar@{}"0";"2"|{\square_1}
  \end{xy}\hspace{20mm}
  \begin{xy}
   (0,0)="0", +(10,0)="1", +(0,-10)="2", +(-10,0)="3"
   \ar@{.}"0";"1" {\color{red}\ar@{-}"1";"2"}\ar@{.}"2";"3" {\color{red}\ar@{-}"3";"0"}\ar@{}"0";"2"|{\square_1}
  \end{xy}
\]\vspace{2mm}\\
 Then the map $\varphi : {\mathbb A}(T) \rightarrow G_1$ of Lemma \ref{lembij} induces a bijection $\varphi : {\mathbb A}(T) \rightarrow {\mathbb P}(G)$ by Lemma \ref{A1A2} (see Example \ref{exmaps}).

 (2) If $n=2$, the following is a complete list of ${\mathbb P}(G)$:
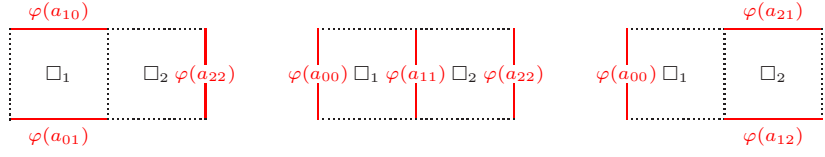
\[
  \begin{xy}
   (0,0)="0", +(13,0)="1", +(0,-12)="2", +(-13,0)="3", "1"+(13,0)="4", +(0,-12)="5"
   {\color{red}\ar@{-}"0";"1"^{\varphi(a_{10})}}\ar@{.}"1";"2" {\color{red}\ar@{-}"2";"3"^{\varphi(a_{01})}}\ar@{.}"3";"0" \ar@{.}"1";"4" {\color{red}\ar@{-}"4";"5"|{\varphi(a_{22})}}\ar@{.}"5";"2" \ar@{}"0";"2"|{\square_1} \ar@{}"1";"5"|{\square_2}
  \end{xy}
  \hspace{15mm}\begin{xy}
   (0,0)="0", +(13,0)="1", +(0,-12)="2", +(-13,0)="3", "1"+(13,0)="4", +(0,-12)="5"
   \ar@{.}"0";"1" {\color{red}\ar@{-}"1";"2"|{\varphi(a_{11})}}\ar@{.}"2";"3"{\color{red}\ar@{-}"3";"0"|{\varphi(a_{00})}}\ar@{.}"1";"4" {\color{red}\ar@{-}"4";"5"|{\varphi(a_{22})}}\ar@{.}"5";"2" \ar@{}"0";"2"|{\square_1} \ar@{}"1";"5"|{\square_2}
  \end{xy}
  \hspace{15mm}\begin{xy}
   (0,0)="0", +(13,0)="1", +(0,-12)="2", +(-13,0)="3", "1"+(13,0)="4", +(0,-12)="5"
   \ar@{.}"0";"1" \ar@{.}"1";"2" \ar@{.}"2";"3" {\color{red}\ar@{-}"3";"0"|{\varphi(a_{00})}}{\color{red}\ar@{-}"1";"4"^{\varphi(a_{21})}}\ar@{.}"4";"5" {\color{red}\ar@{-}"5";"2"^{\varphi(a_{12})}}\ar@{}"0";"2"|{\square_1} \ar@{}"1";"5"|{\square_2}
  \end{xy}
\]\vspace{2mm}\\
 Then the map $\varphi : {\mathbb A}(T) \rightarrow G_1$ of Lemma \ref{lembij} induces a bijection $\varphi : {\mathbb A}(T) \rightarrow {\mathbb P}(G)$ by Lemma \ref{A1A2} (see Example \ref{exmaps}).

 (3) If $n \ge 3$, there are the following cases. A triple $(\square_{i-1},\square_{i},\square_{i+1})$ is called {\it straight} if its squares lie in one column or one row, and {\it zigzag} if not. Triangulations of $(n+3)$-gon are divided into the following $n-1$ types.\\
$\bullet$ Type $k$ $(1 \le k \le n-1)$ : If $v_{n-k}$ is incident to the diagonal $n$ in $T$, $(\square_{n-k-1},\square_{n-k},\square_{n-k+1})$ is straight and $(\square_{n-i-1},\square_{n-i},\square_{n-i+1})$ is zigzag for any $i \in [1,k-1]$.
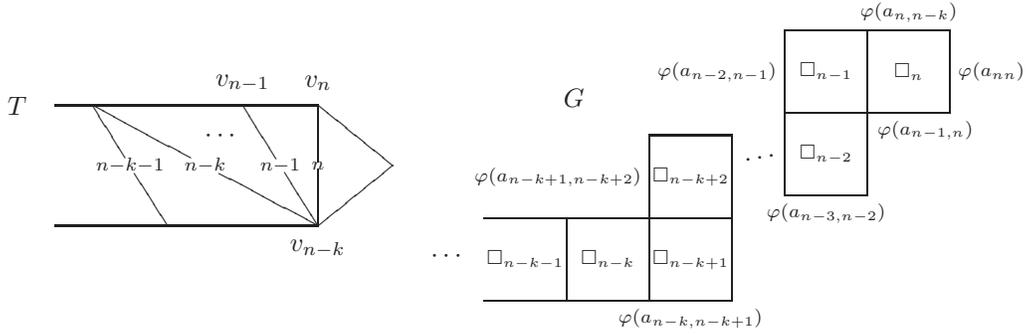
\[
  \begin{xy}
   (0,0)="0", +(5,0)="1", +(20,0)="2", +(10,0)="3", +(10,-8)="4", +(-10,-8)="5", +(-20,0)="6", +(-15,0)="7", "2"+(-3,-4)*{\cdots}, "2"+(0,3)*{v_{n-1}}, "3"+(0,3)*{v_{n}}, "5"+(0,-3)*{v_{n-k}}, "0"+(-5,0)*{T}
   \ar@{-}"0";"3" \ar@{-}"3";"4" \ar@{-}"4";"5" \ar@{-}"5";"7" \ar@{-}"1";"6"|{n-k-1} \ar@{-}"1";"5"|{n-k} \ar@{-}"2";"5"|{n-1} \ar@{-}"3";"5"|{n}
  \end{xy}\hspace{5mm}
  \begin{xy}
   (0,-15)="0", +(11,0)="1", +(11,0)="2", +(0,-11)="3", +(-11,0)="4", +(-11,0)="5", "1"+(0,11)="6", +(11,0)="7", "0"+(-11,0)="8", +(-5,-5)*{\cdots}, "5"+(-11,0)="9", "7"+(7,3)="00", +(11,0)="01", +(0,11)="02", +(-11,0)="03", "02"+(11,0)="04", +(0,-11)="05", "00"+(0,-11)="06", +(11,0)="07", "6"+(-10,5)*{G}, "7"+(4,-3)*{\rotatebox{10}{$\cdots$}}
   \ar@{-}"0";"8" \ar@{-}"5";"9" \ar@{-}"0";"5" \ar@{-}"0";"1" \ar@{-}"1";"2" \ar@{-}"2";"3" \ar@{-}"3";"4"^{\varphi(a_{n-k,n-k+1})} \ar@{-}"4";"5" \ar@{-}"1";"4" \ar@{-}"1";"6"^{\varphi(a_{n-k+1,n-k+2})} \ar@{-}"6";"7" \ar@{-}"7";"2" \ar@{-}"00";"01" \ar@{-}"01";"02" \ar@{-}"02";"03" \ar@{-}"03";"00"_{\varphi(a_{n-2,n-1})} \ar@{-}"02";"04"^{\varphi(a_{n,n-k})} \ar@{-}"04";"05"^{\varphi(a_{nn})} \ar@{-}"01";"05"_(.7){\varphi(a_{n-1,n})} \ar@{-}"00";"06" \ar@{-}"06";"07"_{\varphi(a_{n-3,n-2})} \ar@{-}"07";"01" \ar@{}"03";"01"|{\square_{n-1}} \ar@{}"02";"05"|{\square_{n}} \ar@{}"00";"07"|{\square_{n-2}} \ar@{}"0";"4"|{\square_{n-k}} \ar@{}"1";"3"|{\square_{n-k+1}} \ar@{}"6";"2"|{\square_{n-k+2}} \ar@{}"8";"5"|{\square_{n-k-1}}
  \end{xy}
\]
 Then there is a natural bijection ${\mathbb P}\bigl(G^{[1,n-1]}\bigr) \sqcup {\mathbb P}\bigl(G^{[1,n-k-1]}\bigr) \rightarrow {\mathbb P}(G)$ explained in the following diagram:
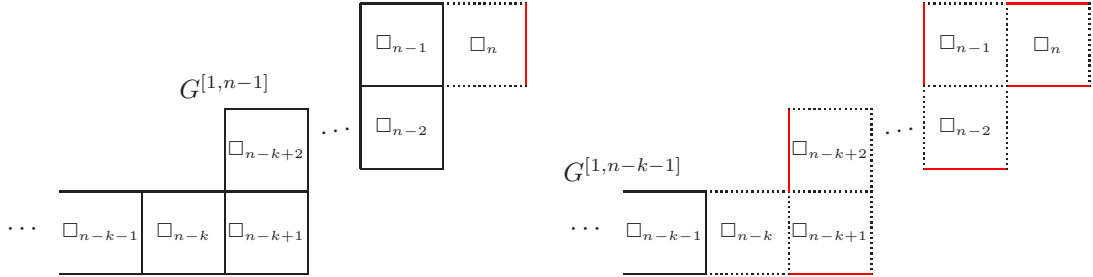
\[
  \begin{xy}
   (0,0)="0", +(11,0)="1", +(11,0)="2", +(0,-11)="3", +(-11,0)="4", +(-11,0)="5", "1"+(0,11)="6", +(11,0)="7", "0"+(-11,0)="8", +(-5,-5)*{\cdots}, "5"+(-11,0)="9", "7"+(7,3)="00", +(11,0)="01", +(0,11)="02", +(-11,0)="03", "02"+(11,0)="04", +(0,-11)="05", "00"+(0,-11)="06", +(11,0)="07", "6"+(0,3)*{G^{[1,n-1]}}, "7"+(4,-3)*{\rotatebox{10}{$\cdots$}}
   \ar@{-}"0";"8" \ar@{-}"5";"9" \ar@{-}"0";"5" \ar@{-}"0";"1" \ar@{-}"1";"2" \ar@{-}"2";"3" \ar@{-}"3";"4" \ar@{-}"4";"5" \ar@{-}"1";"4" \ar@{-}"1";"6" \ar@{-}"6";"7" \ar@{-}"7";"2" \ar@{-}"00";"01" \ar@{-}"01";"02" \ar@{-}"02";"03" \ar@{-}"03";"00" \ar@{.}"02";"04" {\color{red}\ar@{-}"04";"05"}\ar@{.}"01";"05" \ar@{-}"00";"06" \ar@{-}"06";"07" \ar@{-}"07";"01" \ar@{}"03";"01"|{\square_{n-1}} \ar@{}"02";"05"|{\square_{n}} \ar@{}"00";"07"|{\square_{n-2}} \ar@{}"0";"4"|{\square_{n-k}} \ar@{}"1";"3"|{\square_{n-k+1}} \ar@{}"6";"2"|{\square_{n-k+2}} \ar@{}"8";"5"|{\square_{n-k-1}}
  \end{xy}\hspace{5mm}
  \begin{xy}
   (0,0)="0", +(11,0)="1", +(11,0)="2", +(0,-11)="3", +(-11,0)="4", +(-11,0)="5", "1"+(0,11)="6", +(11,0)="7", "0"+(-11,0)="8", +(-5,-5)*{\cdots}, "5"+(-11,0)="9", "7"+(7,3)="00", +(11,0)="01", +(0,11)="02", +(-11,0)="03", "02"+(11,0)="04", +(0,-11)="05", "00"+(0,-11)="06", +(11,0)="07", "8"+(0,3)*{G^{[1,n-k-1]}}, "7"+(4,-3)*{\rotatebox{10}{$\cdots$}}
   \ar@{-}"0";"8" \ar@{-}"5";"9" \ar@{-}"0";"5" \ar@{.}"0";"1" \ar@{.}"1";"2" \ar@{.}"2";"3" {\color{red}\ar@{-}"3";"4"}\ar@{.}"4";"5" \ar@{.}"1";"4" {\color{red}\ar@{-}"1";"6"}\ar@{.}"6";"7" \ar@{.}"7";"2" \ar@{.}"00";"01" \ar@{.}"01";"02" \ar@{.}"02";"03" {\color{red}\ar@{-}"03";"00"}{\color{red}\ar@{-}"02";"04"}\ar@{.}"04";"05" {\color{red}\ar@{-}"01";"05"}\ar@{.}"00";"06"{\color{red}\ar@{-}"06";"07"}\ar@{.}"07";"01" \ar@{}"03";"01"|{\square_{n-1}} \ar@{}"02";"05"|{\square_{n}} \ar@{}"00";"07"|{\square_{n-2}} \ar@{}"0";"4"|{\square_{n-k}} \ar@{}"1";"3"|{\square_{n-k+1}} \ar@{}"6";"2"|{\square_{n-k+2}} \ar@{}"8";"5"|{\square_{n-k-1}}
  \end{xy}
\]
 Note that, in type $n-1$, $(\square_{i-1},\square_{i},\square_{i+1})$ is zigzag for any $i \in [2,n-1]$.
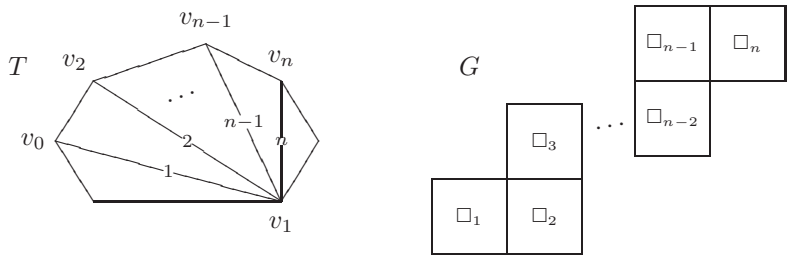
\[
  \begin{xy}
   (0,0)="0", +(5,8)="1", +(15,5)="2", +(10,-5)="3", +(5,-8)="4", +(-5,-8)="5", +(-25,0)="6", "2"+(-3,-7)*{\rotatebox{20}{$\cdots$}}, "0"+(-3,0)*{v_0}, "1"+(-2.5,2.5)*{v_2}, "2"+(0,3)*{v_{n-1}}, "3"+(0,3)*{v_{n}}, "5"+(0,-3)*{v_1}, "0"+(-5,10)*{T}
   \ar@{-}"1";"2" \ar@{-}"2";"3" \ar@{-}"3";"4" \ar@{-}"4";"5" \ar@{-}"5";"6" \ar@{-}"0";"1" \ar@{-}"0";"6" \ar@{-}"0";"5"|{1} \ar@{-}"1";"5"|{2} \ar@{-}"2";"5"|{n-1} \ar@{-}"3";"5"|{n}
  \end{xy}\hspace{15mm}
  \begin{xy}
   (0,-5)="0", +(10,0)="1", +(10,0)="2", +(0,-10)="3", +(-10,0)="4", +(-10,0)="5", "1"+(0,10)="6", +(10,0)="7", "7"+(7,3)="00", +(10,0)="01", +(0,10)="02", +(-10,0)="03", "02"+(10,0)="04", +(0,-10)="05", "00"+(0,-10)="06", +(10,0)="07", "6"+(-5,5)*{G}, "7"+(4,-3)*{\rotatebox{10}{$\cdots$}}
   \ar@{-}"0";"5" \ar@{-}"0";"1" \ar@{-}"1";"2" \ar@{-}"2";"3" \ar@{-}"3";"4" \ar@{-}"4";"5" \ar@{-}"1";"4" \ar@{-}"1";"6" \ar@{-}"6";"7" \ar@{-}"7";"2" \ar@{-}"00";"01" \ar@{-}"01";"02" \ar@{-}"02";"03" \ar@{-}"03";"00" \ar@{-}"02";"04" \ar@{-}"04";"05" \ar@{-}"01";"05" \ar@{-}"00";"06" \ar@{-}"06";"07" \ar@{-}"07";"01" \ar@{}"03";"01"|{\square_{n-1}} \ar@{}"02";"05"|{\square_{n}} \ar@{}"00";"07"|{\square_{n-2}} \ar@{}"0";"4"|{\square_{1}} \ar@{}"1";"3"|{\square_{2}} \ar@{}"6";"2"|{\square_{3}}
  \end{xy}
\]
 In this case, we define $G^{[1,0]}$ to be the graph consisting of the edge of $\square_1$ not incident to $\square_2$. Then there is also a natural bijection ${\mathbb P}\bigl(G^{[1,n-1]}\bigr) \sqcup {\mathbb P}\bigl(G^{[1,0]}\bigr) \rightarrow {\mathbb P}(G)$ explained in the following diagram:
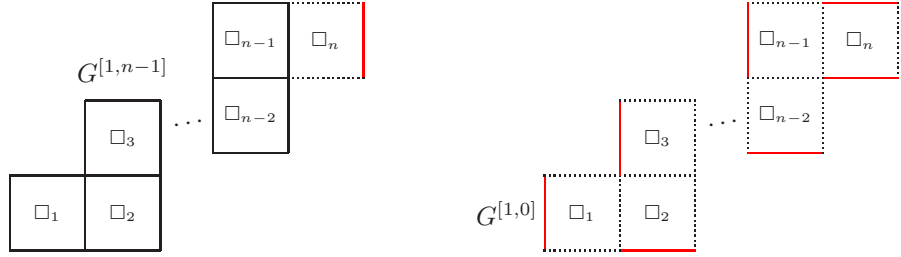
\[
  \begin{xy}
   (0,0)="0", +(10,0)="1", +(10,0)="2", +(0,-10)="3", +(-10,0)="4", +(-10,0)="5", "1"+(0,10)="6", +(10,0)="7", "7"+(7,3)="00", +(10,0)="01", +(0,10)="02", +(-10,0)="03", "02"+(10,0)="04", +(0,-10)="05", "00"+(0,-10)="06", +(10,0)="07", "6"+(5,4)*{G^{[1,n-1]}}, "7"+(4,-3)*{\rotatebox{10}{$\cdots$}}
   \ar@{-}"0";"5" \ar@{-}"0";"1" \ar@{-}"1";"2" \ar@{-}"2";"3" \ar@{-}"3";"4" \ar@{-}"4";"5" \ar@{-}"1";"4" \ar@{-}"1";"6" \ar@{-}"6";"7" \ar@{-}"7";"2" \ar@{-}"00";"01" \ar@{-}"01";"02" \ar@{-}"02";"03" \ar@{-}"03";"00" \ar@{.}"02";"04" {\color{red}\ar@{-}"04";"05"}\ar@{.}"01";"05" \ar@{-}"00";"06" \ar@{-}"06";"07" \ar@{-}"07";"01" \ar@{}"03";"01"|{\square_{n-1}} \ar@{}"02";"05"|{\square_{n}} \ar@{}"00";"07"|{\square_{n-2}} \ar@{}"0";"4"|{\square_{1}} \ar@{}"1";"3"|{\square_{2}} \ar@{}"6";"2"|{\square_{3}}
  \end{xy}
  \hspace{15mm}
  \begin{xy}
   (0,0)="0", +(10,0)="1", +(10,0)="2", +(0,-10)="3", +(-10,0)="4", +(-10,0)="5", "1"+(0,10)="6", +(10,0)="7", "7"+(7,3)="00", +(10,0)="01", +(0,10)="02", +(-10,0)="03", "02"+(10,0)="04", +(0,-10)="05", "00"+(0,-10)="06", +(10,0)="07", "0"+(-5,-5)*{G^{[1,0]}}, "7"+(4,-3)*{\rotatebox{10}{$\cdots$}}
   {\color{red}\ar@{-}"0";"5"}\ar@{.}"0";"1" \ar@{.}"1";"2" \ar@{.}"2";"3" {\color{red}\ar@{-}"3";"4"}\ar@{.}"4";"5" \ar@{.}"1";"4" {\color{red}\ar@{-}"1";"6"}\ar@{.}"6";"7" \ar@{.}"7";"2" \ar@{.}"00";"01" \ar@{.}"01";"02" \ar@{.}"02";"03" {\color{red}\ar@{-}"03";"00"}{\color{red}\ar@{-}"02";"04"}\ar@{.}"04";"05" {\color{red}\ar@{-}"01";"05"}\ar@{.}"00";"06" {\color{red}\ar@{-}"06";"07"}\ar@{.}"07";"01" \ar@{}"03";"01"|{\square_{n-1}} \ar@{}"02";"05"|{\square_n} \ar@{}"00";"07"|{\square_{n-2}} \ar@{}"0";"4"|{\square_{1}} \ar@{}"1";"3"|{\square_{2}} \ar@{}"6";"2"|{\square_{3}}
  \end{xy}
\]


\begin{proof}[Proof of Proposition \ref{PMPM}]
 We prove it by induction on $n$. For $n=1, 2$, the assertion follows from the above observations (1) and (2). Assume $n \ge 3$. Then ${\mathbb A}(T)$ is written as a disjoint union ${\mathbb A}(T)={\mathbb A}'(T) \sqcup {\mathbb A}''(T)$, where ${\mathbb A}'(T)$ consists of all $A \in {\mathbb A}(T)$ containing $a_{nn}$. Then the natural inclusion $A(T^{[1,n-1]}) \rightarrow A(T)$ induces a bijection ${\mathbb A}(T^{[1,n-1]}) \rightarrow {\mathbb A}'(T)$ given by $A \mapsto A \sqcup \{a_{nn}\}$. We consider $n-1$ types in the above observation (3).\\
$\bullet$ Type $k$ ($1 \le k \le n-2$) : In this case, any $A \in {\mathbb A}''(T)$ contains $a_{n,n-k}$ and $a_{n-1,n}$ by Lemma \ref{A1A2}.
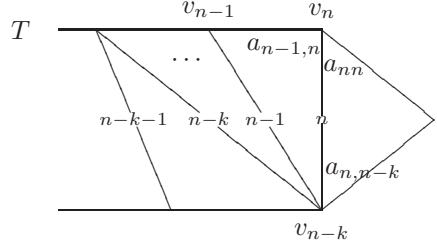
\[
  \begin{xy}
   (0,0)="0", +(5,0)="1", +(15,0)="2", +(15,0)="3", +(15,-12)="4", +(-15,-12)="5", +(-20,0)="6", +(-15,0)="7", "2"+(-3,-4)*{\cdots}, "2"+(0,2.5)*{v_{n-1}}, "3"+(0,2.5)*{v_{n}}, "5"+(0,-3)*{v_{n-k}}, "0"+(-5,0)*{T}, "3"+(-5,-2.5)*{a_{n-1,n}}, "3"+(3,-5)*{a_{nn}}, "5"+(5.5,5)*{a_{n,n-k}}
   \ar@{-}"0";"3" \ar@{-}"3";"4" \ar@{-}"4";"5" \ar@{-}"5";"7" \ar@{-}"1";"6"|{n-k-1} \ar@{-}"1";"5"|{n-k} \ar@{-}"2";"5"|{n-1} \ar@{-}"3";"5"|{n}
  \end{xy}
\]
 Thus the natural inclusion $A(T^{[1,n-k-1]}) \rightarrow A(T)$ induces a bijection ${\mathbb A}(T^{[1,n-k-1]}) \rightarrow {\mathbb A}''(T)$ given by $A \mapsto A \sqcup \{a_{n,n-k}, a_{n-1,n}, a_{n-2,n-1},\ldots,a_{n-k,n-k+1}\}$.
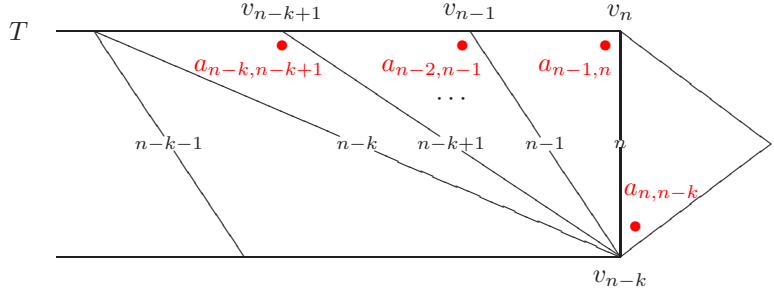
\[
  \begin{xy}
   (0,0)="-1", +(5,0)="0", +(25,0)="1", +(25,0)="2", +(20,0)="3", +(20,-15)="4", +(-20,-15)="5", +(-50,0)="6", +(-25,0)="7", "2"+(-2.5,-9)*{\cdots}, "2"+(0,2.5)*{v_{n-1}}, "3"+(0,2.5)*{v_{n}}, "5"+(0,-3)*{v_{n-k}}, "1"+(0,2.5)*{v_{n-k+1}}, "-1"+(-5,0)*{T}, "3"+(-2,-2)*{{\color{red}\bullet}}, +(-4,-3)*{{\color{red}a_{n-1,n}}}, "5"+(2,4)*{{\color{red}\bullet}}, +(3.5,4.5)*{{\color{red}a_{n,n-k}}}, "2"+(-1,-2)*{{\color{red}\bullet}}, +(-4,-3)*{{\color{red}a_{n-2,n-1}}}, "1"+(0,-2)*{{\color{red}\bullet}}, +(-3,-3)*{{\color{red}a_{n-k,n-k+1}}}
   \ar@{-}"-1";"3" \ar@{-}"3";"4" \ar@{-}"4";"5" \ar@{-}"5";"7" \ar@{-}"0";"6"|{n-k-1} \ar@{-}"0";"5"|{n-k} \ar@{-}"1";"5"|{n-k+1} \ar@{-}"2";"5"|{n-1} \ar@{-}"3";"5"|{n}
  \end{xy}
\]
 By induction on $n$, $\varphi$ induces a bijection
 \begin{equation*}
  {\mathbb A}(T) \simeq {\mathbb A}(T^{[1,n-1]}) \sqcup {\mathbb A}(T^{[1,n-k-1]}) \simeq {\mathbb P}(G^{[1,n-1]}) \sqcup {\mathbb P}(G^{[1,n-k-1]}) \simeq {\mathbb P}(G).
 \end{equation*}
 By construction, it is easy to check that it is again compatible with $\varphi$.\\
$\bullet$ Type $n-1$ : In this case, any element of ${\mathbb A}''(T)$ contains $a_{n1}$ and $a_{n-1,n}$, thus ${\mathbb A}''(T)=\{A_0:=\{a_{00},a_{n1},a_{12},\ldots,a_{n-1,n}\}\}$ by Lemma \ref{A1A2}.
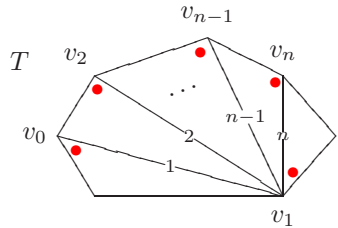
\[
\begin{xy}
 (0,0)="0", +(5,8)="1", +(15,5)="2", +(10,-5)="3", +(7,-8)="4", +(-7,-8)="5", +(-25,0)="6", "2"+(-3,-7)*{\rotatebox{20}{$\cdots$}}, "0"+(-3,0)*{v_0}, "1"+(-2.5,2.5)*{v_2}, "2"+(0,3)*{v_{n-1}}, "3"+(0,3)*{v_{n}}, "5"+(0,-3)*{v_1}, "0"+(-5,10)*{T}, "5"+(1.3,3)*{{\color{red}\bullet}}, "3"+(-1,-1)*{{\color{red}\bullet}}, "2"+(-1,-2)*{{\color{red}\bullet}}, "1"+(0.3,-2)*{{\color{red}\bullet}}, "0"+(2.5,-2)*{{\color{red}\bullet}}
 \ar@{-}"1";"2" \ar@{-}"2";"3" \ar@{-}"3";"4" \ar@{-}"4";"5" \ar@{-}"5";"6" \ar@{-}"0";"1" \ar@{-}"0";"6" \ar@{-}"0";"5"|{1} \ar@{-}"1";"5"|{2} \ar@{-}"2";"5"|{n-1} \ar@{-}"3";"5"|{n}
\end{xy}
\]
 Then $\varphi$ induces a bijection
\begin{equation*}
 {\mathbb A}(T) \simeq {\mathbb A}(T^{[1,n-1]}) \sqcup \{A_0\} \simeq {\mathbb P}(G^{[1,n-1]}) \sqcup {\mathbb P}\bigl(G^{[1,0]}\bigr) \simeq {\mathbb P}(G).
\end{equation*}
 By construction, it is easy to check that it is again compatible with $\varphi$.
\end{proof}

\begin{proof}[Proof of Theorem \ref{angle}]
 The assertion follows immediately from Theorem \ref{MS} and Proposition \ref{PMPM}.
\end{proof}



\subsection{Proof of Corollary \ref{main1}}

 We have a natural bijection $\rho : A(T) \rightarrow \overline{Q}_1$ given by the following picture:
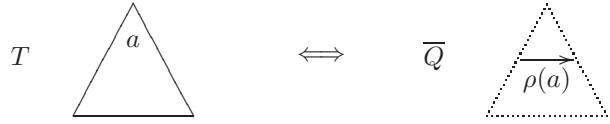
\[
  \begin{xy}
   (0,0)="1", +(15,0)="2", +(15,0)="3", +(0,-15)="4", +(-7,0)="5", +(-16,0)="6", +(-7,0)="7", "1"+(0,-7)*{T}, "3"+(10,-7)*{\Longleftrightarrow}, +(15,0)*{\overline{Q}}, "3"+(25,0)="01", +(15,0)="02", +(15,0)="03", +(0,-15)="04", +(-7,0)="05", +(-16,0)="06", +(-7,0)="07", "2"+(0,-5)*{a}, "02"+(-3.5,-7.5)="A", +(7,0)="B"
   \ar@{-}"5";"6" \ar@{-}"2";"6" \ar@{-}"2";"5" \ar@{.}"05";"06" \ar@{.}"02";"06" \ar@{.}"02";"05" \ar@{->}"A";"B"_*{\rho(a)}
  \end{xy}
\]
We denote $\rho(a_{ij})$ by $\alpha_{ij}$.

 Moreover, we denote other arrows in this triangle by the following, if exists:
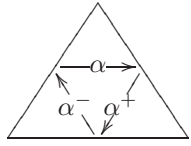
\[
  \begin{xy}
   (0,0)="01", +(18,0)="02", +(18,0)="03", +(0,-18)="04", +(-6,0)="05", +(-24,0)="06", +(-6,0)="07"
   \ar@{-}"05";"06" \ar@{-}"02";"06" \ar@{-}"02";"05" \ar@{->}"02"+(-5,-8.5);"02"+(5,-8.5)|*{\alpha} \ar@{->}"02"+(5.5,-9.5);"06"+(12.5,0.5)|*{\alpha^+} \ar@{->}"06"+(11.5,0.5);"02"+(-5.5,-9.5)|*{\alpha^-}
  \end{xy}
\]


\begin{lemma}\label{pipm}
 For any $A \in {\mathbb A}(T)$, we have $\rho(A) \in {\mathbb D}(\overline{Q})$.
\end{lemma}

\begin{proof}
 Firstly, we show that $\rho(A)$ is a discrete subset of $\overline{Q}$. Assume that $a_{ij} \neq a_{hk}$ are two elements of $A$ such that there is a path $\alpha_{il}\alpha_{i+1,l}\cdots \alpha_{h-1,l}\alpha_{hl}$ in $Q$ from $t(\alpha_{ij})$ to $s(\alpha_{hk})$. Since $a_{ij}, a_{hk} \in A$, $i \neq h$ and $j \neq k$ hold by Lemma \ref{A1A2}. It follows from $j \neq k$ that at least one of $\alpha_{ij}$ and $\alpha_{hk}$ does not belong to $Q$. By symmetry, we assume that $t(\alpha_{hk})$ is a boundary arc of $T$. Then $\alpha_{ij}$ is either $\alpha_{il}$ or $\alpha_{il}^-$.

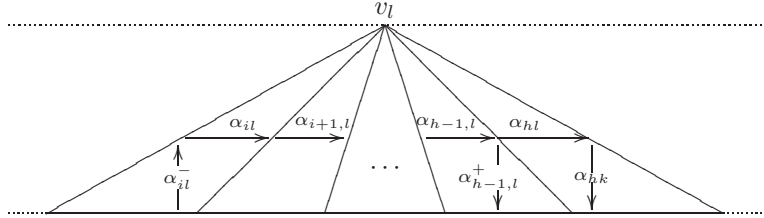
\[
  \begin{xy}
   (0,0)="0", +(50,0)="1", +(50,0)="2", "0"+(5,-25)="3", +(20,0)="4", +(17,0)="5", +(16,0)="6", +(17,0)="7", +(20,0)="8", "5"+(8,6)*{\cdots}, "1"+(0,2)*{v_l}, "3"+(-5,0)="A", "8"+(5,0)="B"
   \ar@{.}"0";"2" \ar@{-}"3";"8" \ar@{-}"1";"3" \ar@{-}"1";"4" \ar@{-}"1";"5" \ar@{-}"1";"6" \ar@{-}"1";"7" \ar@{-}"1";"8" \ar@{.}"A";"3" \ar@{.}"8";"B" \ar@{->}"1"+(-26.5,-15);"1"+(-15.5,-15)^(.7){\alpha_{il}} \ar@{->}"1"+(-14.5,-15);"1"+(-5.5,-15)^(.7){\alpha_{i+1,l}} \ar@{->}"1"+(5.5,-15);"1"+(14.5,-15)^(.3){\alpha_{h-1,l}} \ar@{->}"1"+(15,-15);"1"+(27,-15)^(.3){\alpha_{hl}} \ar@{->}"8"+(-17.5,9);"8"+(-17.5,0.5)|{\alpha_{hk}} \ar@{->}"3"+(17.5,0.5);"3"+(17.5,9)|{\alpha_{il}^-} \ar@{->}"7"+(-10,9);"7"+(-10,0.5)|{\alpha_{h-1,l}^+ \ \ }
  \end{xy}
\]
 Since $\alpha_{hl}^-$ does not belong to $\rho(A)$ and $A$ is a perfect matching, $\alpha_{h-1,l}^+$ belongs to $\rho(A)$. Repeating the same argument, $\alpha_{sl}^+$ belongs to $\rho(A)$ for any $s \in [i,h]$. This is a contradiction since $\alpha_{ij}$ and $\alpha_{il}^+$ belong to $\rho(A)$.

 To prove that $\rho(A)$ is maximal discrete, take $a_{ij} \in A(T) \setminus A$. By (A1), there exists $a_{ik} \in A$ with $k \neq j$. In this case, there exists a path in $Q$ of length $0$ or $1$ either from $t(\alpha_{ij})$ to $s(\alpha_{ik})$, or from $t(\alpha_{ik})$ to $s(\alpha_{ij})$. In both cases, $\rho(A) \sqcup \{\alpha_{ij}\}$ is not discrete. Thus the assertion follows.
\end{proof}


\begin{proposition}\label{DPM}
 The bijection $\rho : A(T) \rightarrow \overline{Q}_1$ induces a bijection $\rho : {\mathbb A}(T) \rightarrow {\mathbb D}(\overline{Q})$.
\end{proposition}

\begin{proof}
 By Lemma \ref{pipm}, we only have to show that for any $D \in {\mathbb D}(\overline{Q})$, we have $\rho^{-1}(D) \in {\mathbb A}(T)$.

 Firstly, we assume that $\rho^{-1}(D)$ does not satisfy (A1). Then there exists $i$ such that one of the following conditions holds.

(1) There are two elements $\alpha_{ij} \neq \alpha_{ik}$ in $D$.

(2) There are no $j$ such that $\alpha_{ij} \in D$.

 In the case (1), there exists a path in $Q$ of length $0$ or $1$ either from $t(\alpha_{ij})$ to $s(\alpha_{ik})$, or from $t(\alpha_{ik})$ to $s(\alpha_{ij})$. Thus $D$ is not discrete.

 In the case (2), we take an arrow $\alpha_{ij}$ such that $s(\alpha_{ij}^-)=t(\alpha_{ij}^+)$ is a boundary arc of $T$. Since $D$ is maximal discrete and $\alpha_{ij} \notin D$, one of the following conditions hold.

(2-i) There exists $\alpha \in D$ such that there is a path in $Q$ from $t(\alpha_{ij})$ to $s(\alpha)$.

(2-ii)  There exists $\alpha \in D$ such that there is a path in $Q$ from $t(\alpha)$ to $s(\alpha_{ij})$.

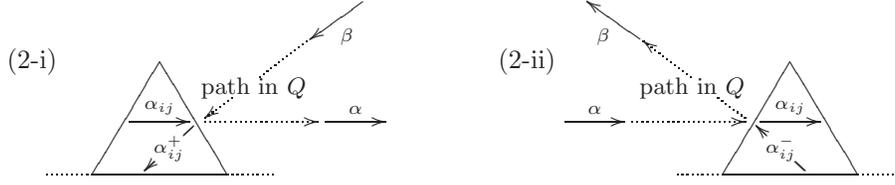
\[
  \begin{xy}
   (0,0)="01", +(15,0)="02", +(15,0)="03", +(0,-15)="04", +(-6,0)="05", +(-18,0)="06", +(-6,0)="07", "03"+(7,-8)="A", +(8,0)="B", "01"+(-2,0)*{\mbox{(2-i)}}, "02"+(20,3)="C", +(7,5)="D"
   \ar@{.}"04";"05" \ar@{-}"05";"06" \ar@{.}"06";"07" \ar@{-}"02";"06" \ar@{-}"02";"05" \ar@{->}"02"+(-4,-8);"02"+(4,-8)^{\alpha_{ij}} \ar@{->}"05"+(-4.5,6.5);"05"+(-11,0.5)|{\alpha_{ij}^+} \ar@{->}"A";"B"^{\alpha} \ar@{.>}"02"+(6,-8);"A"+(-1,0) \ar@{->}"D";"C"^{\beta} \ar@{.>}"C"+(-0.5,-0.3);"02"+(6,-7.5)|(.6){\mbox{{\small \ \  path in $Q$}}}
  \end{xy}\hspace{15mm}
  \begin{xy}
   (0,0)="01", +(15,0)="02", +(15,0)="03", +(0,-15)="04", +(-6,0)="05", +(-18,0)="06", +(-6,0)="07", "01"+(-7,-8)="A", +(-8,0)="B", "01"+(-20,0)*{\mbox{(2-ii)}}, "02"+(-20,3)="C", +(-7,5)="D"
   \ar@{.}"04";"05" \ar@{-}"05";"06" \ar@{.}"06";"07" \ar@{-}"02";"06" \ar@{-}"02";"05" \ar@{->}"02"+(-4,-8);"02"+(4,-8)^{\alpha_{ij}} \ar@{->}"06"+(11,0.5);"06"+(4.5,6.5)|{\alpha_{ij}^-} \ar@{->}"B";"A"^{\alpha} \ar@{.>}"A"+(1,0);"02"+(-6,-8) \ar@{->}"C";"D"^{\beta} \ar@{.>}"02"+(-6,-7.5);"C"+(0.5,-0.3)|(.4){\mbox{{\small path in $Q$ \ \ }}}
  \end{xy}
\]

 We consider the case (2-i). In this case, there exists the arrow $\alpha_{ij}^+$. Since $D$ is maximal discrete and $\alpha_{ij}^+ \notin D$, there exists $\beta \in D$ such that there is a path in $Q$ from $t(\beta)$ to $s(\alpha)$ factoring through $s(\alpha_{ij}^+)=t(\alpha_{ij})$. Thus there is a path in $Q$ from $t(\beta)$ to $s(\alpha)$, a contradiction. In the case (2-ii), we have a contradiction by the same argument. Consequently, $\rho^{-1}(D)$ satisfies (A1).

 Secondly, we assume that $\rho^{-1}(D)$ does not satisfy (A2). Then there exists $j$ such that one of the following conditions holds.

(1) There are two elements $\alpha_{ij} \neq \alpha_{hj}$ in $D$.

(2) There are no $i$ such that $\alpha_{ij} \in D$.

 In the case (1), there is a path in $Q$ either from $t(\alpha_{ij})$ to $s(\alpha_{hj})$, or from $t(\alpha_{hj})$ to $s(\alpha_{ij})$. Thus $\{\alpha_{ij},\alpha_{hj}\}$ is not discrete.

 In the case (2), all angles at the vertex $v_j$ are labeled as $a_{sj}, a_{s+1,j}, \ldots, a_{t-1,j}, a_{tj}$.
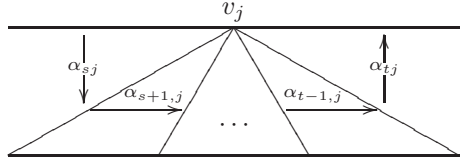
\[
  \begin{xy}
   (0,0)="0", +(30,0)="1", +(30,0)="2", "0"+(0,-17)="3", +(20,0)="4", +(20,0)="5", +(20,0)="6", "1"+(0,-13)*{\cdots}, "1"+(0,2)*{v_j}
   \ar@{-}"0";"2" \ar@{-}"3";"6" \ar@{-}"1";"3" \ar@{-}"1";"4" \ar@{-}"1";"5" \ar@{-}"1";"6" \ar@{->}"1"+(-20,-1);"1"+(-20,-10)|{\alpha_{sj}} \ar@{->}"1"+(-19,-11);"1"+(-7,-11)^(.7){\alpha_{s+1,j}} \ar@{->}"1"+(20,-10);"1"+(20,-1)|{\alpha_{tj}} \ar@{->}"1"+(7,-11);"1"+(19,-11)^(.3){\alpha_{t-1,j}}
  \end{xy}
\]

 Since $\rho^{-1}(D)$ satisfies (A1), for any $i \in [s,t]$, precisely one of $\alpha_{ij}^-$ and $\alpha_{ij}^+$ belongs to $D$. Assume $\alpha_{ij}^+ \in D$ for some $i \in [s,t-1]$. Then $\alpha_{i+1,j}^-$ does not belong to $D$ since $s(\alpha_{ij}^+)=t(\alpha_{i+1,j}^-)$. Thus $\alpha_{i+1,j}^+$ exists and belongs to $D$. Repeating the same argument, $\alpha_{hj}^+$ exists and belongs to $D$ for any $h \in [i,t]$. This is a contradiction since the arrow $\alpha_{tj}^- : t(\alpha_{tj}^+) \rightarrow s(\alpha_{t-1,j}^+)$ belongs to $Q$.
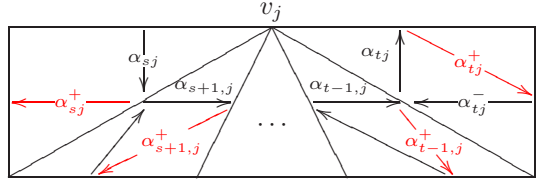
\[
  \begin{xy}
   (0,0)="0", +(35,0)="1", +(35,0)="2", "0"+(0,-20)="3", +(25,0)="4", +(20,0)="5", +(25,0)="6", "1"+(0,2)*{v_j}, "1"+(0,-13)*{\cdots}
   \ar@{-}"0";"2" \ar@{-}"3";"6" \ar@{-}"1";"3" \ar@{-}"1";"4" \ar@{-}"1";"5" \ar@{-}"1";"6" \ar@{-}"0";"3" \ar@{-}"2";"6" \ar@{->}"1"+(-17,-0.5);"1"+(-17,-8.5)|{\alpha_{sj}} \ar@{->}"1"+(-17,-10);"1"+(-5.5,-10)^(.7){\alpha_{s+1,j}} {\color{red}\ar@{->}"1"+(-19,-10);"1"+(-34.5,-10)|{\alpha_{sj}^+}}\ar@{->}"1"+(-24,-19.5);"1"+(-17,-11) {\color{red}\ar@{->}"1"+(-6,-11);"1"+(-23,-19.5)|{\ \ \alpha_{s+1,j}^+}}\ar@{->}"1"+(17,-8.5);"1"+(17,-0.5)^(.6){\alpha_{tj}} \ar@{->}"1"+(5.5,-10);"1"+(17,-10)^(.3){\alpha_{t-1,j}} \ar@{->}"1"+(34.5,-10);"1"+(19,-10)|{\alpha_{tj}^-} {\color{red}\ar@{->}"1"+(17,-11);"1"+(24,-19.5)|{\ \ \alpha_{t-1,j}^+}}\ar@{->}"1"+(23,-19.5);"1"+(6,-11) {\color{red}\ar@{->}"1"+(18,-0.5);"1"+(34.5,-9)|{\alpha_{tj}^+}}
  \end{xy}
\]
Assume $\alpha_{ij}^- \in D$ for some $i \in [s+1,t]$. Then we have a contradiction by the same argument. 

 It remains to consider the case $s+1=t$, and $\alpha_{sj}^-$ and $\alpha_{s+1,j}^+$ belong to $D$. In this case, there is a path $\alpha_{s+1,j}^- \alpha_{sj}^+ : t(\alpha_{s+1,j}^+) \rightarrow s(\alpha_{sj}^-)$ in $Q$, a contradiction.
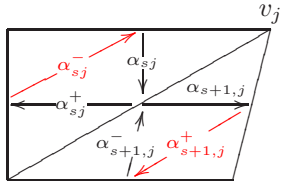
\[
  \begin{xy}
   (0,0)="0", +(35,0)="1", "0"+(0,-20)="3", +(30,0)="4", "1"+(0,2)*{v_j}
   \ar@{-}"0";"1" \ar@{-}"3";"4" \ar@{-}"1";"3" \ar@{-}"1";"4" \ar@{-}"0";"3" \ar@{->}"1"+(-17,-0.5);"1"+(-17,-8.5)|{\alpha_{sj}} \ar@{->}"1"+(-17,-10);"1"+(-3,-10)^(.7){\alpha_{s+1,j}} {\color{red}\ar@{->}"1"+(-34.5,-9);"1"+(-18,-0.5)|{\alpha_{sj}^-}}\ar@{->}"1"+(-18.5,-10);"1"+(-34.5,-10)|{\alpha_{sj}^+} \ar@{->}"1"+(-19,-19.5);"1"+(-17,-11)|{\alpha_{s+1,j}^- \ \ } {\color{red}\ar@{->}"1"+(-4,-11);"1"+(-18,-19.5)|{\ \ \alpha_{s+1,j}^+}}
  \end{xy}
\]
 Consequently, $\rho^{-1}(D)$ satisfies (A2).
\end{proof}


\begin{proof}[Proof of Corollary \ref{main1}]
 The assertion follows immediately from Theorem \ref{angle} and Proposition \ref{DPM}.
\end{proof}




\section{Minimal cuts of quivers with potential}

 In this section, for a quiver $Q$ of type $A_n$, we show that maximal discrete subsets of $\overline{Q}$ coincides with {\it minimal cuts} of the {\it quiver with potential} ($\widetilde{Q},\widetilde{W}$) introduced by Demonet and Luo \cite{DL}.


\subsection{Quivers with potential and cuts}

 We recall the definitions of quivers with potential \cite{DWZ} and of their cuts \cite{BFPPT,HI}. For a quiver $R$, we denote by ${\mathbb Z}R$ the path algebra of $R$ over the ring ${\mathbb Z}$ of integers.

\begin{definition}
 (1) A {\it quiver with potential} (QP for short) is a pair $(R,W)$ of a quiver $R$ and an element $W \in {\mathbb Z}R$ which is a linear combination of cyclic paths.\par
 (2) A {\it cut} of a QP $(R,W)$ is a subset $C$ of $R_1$ such that any cyclic path appearing in $W$ contains a precisely one arrow in $C$.
\end{definition}

 Let us recall the QP $(\widetilde{Q},\widetilde{W})$ introduced in \cite{DL}. The quiver $\widetilde{Q}$ has the set of vertices $\widetilde{Q}_0=\overline{Q}_0$ and the set of arrows $\widetilde{Q}_1=\overline{Q}_1 \sqcup I \sqcup E$ defined as follows:\\
$\bullet$ $I$ consists of arrows from $i$ to $j$, where $i$ and $j$ are boundary arcs that are in a common triangle of $T$ and $j$ follows $i$ in the clockwise order.\\
$\bullet$ $E$ consists of arrows from $i$ to $j$ such that $i$ and $j$ are boundary arcs that are not in a common triangle of $T$ and $i$ is a predecessor of $j$ with respect to anti-clockwise order. An element of $E$ is called an {\it external arrow}.\par

 We consider the following two types of cycles of $\widetilde{Q}$. A {\it triangle cycle} is a cycle of length $3$ inside a triangle of $T$. A {\it big cycle} is a cycle which contains precisely one external arrow, winding around a vertex of the polygon. We define
\begin{equation*}
 \widetilde{W}=\sum(\mbox{triangle cycles in }\widetilde{Q})-\sum(\mbox{big cycles in }\widetilde{Q}).
\end{equation*}
 By construction, it is easy to see that both the number of triangle cycles in $\widetilde{Q}$ and the number of big cycles in $\widetilde{Q}$ are $n+1$.\par

\begin{example}\label{A2cut}
 Up to rotation, all triangulations of a pentagon give the following situation:
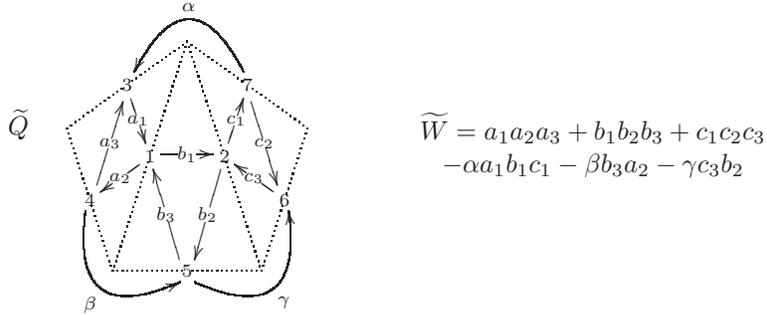
\[
  \widetilde{Q}\hspace{5mm}
  \begin{xy}
   (0,0)="0", +(16,11.5)="1", +(16,-11.5)="2", +(-6.1,-19)="3", +(-19.8,0)="4"
   \ar@{.}"0";"1"|{3} \ar@{.}"1";"2"|{7} \ar@{.}"2";"3"|{6} \ar@{.}"3";"4"|{5} \ar@{.}"4";"0"|{4} \ar@{.}"1";"4"|{1} \ar@{.}"1";"3"|{2} \ar@{->}"1"+(-7.5,-7.5);"1"+(-5.5,-13.5)|{a_1} \ar@{->}"1"+(-3.5,-15);"1"+(3.5,-15)|{b_1} \ar@{->}"1"+(5.5,-13.5);"1"+(7.5,-7.5)|{c_1} \ar@{->}"1"+(-6.3,-16.3);"1"+(-11.5,-20)|{a_2} \ar@{->}"1"+(-12,-19);"1"+(-8.5,-8)|{a_3} \ar@{->}"1"+(4.5,-17);"1"+(1,-29)|{b_2} \ar@{->}"1"+(-1,-29);"1"+(-4.5,-17)|{b_3} \ar@{->}"1"+(8.5,-8);"1"+(12,-19)|{c_2} \ar@{->}"1"+(11.5,-20);"1"+(6.3,-16.3)|{c_3} \ar@{->}@/^-7mm/"1"+(7.5,-4);"1"+(-7,-4)_{\alpha} \ar@{->}@/^-7mm/"1"+(-13.5,-22.5);"1"+(-1,-32)_{\beta} \ar@{->}@/^-7mm/"1"+(1,-32);"1"+(13.5,-22.5)_{\gamma}
  \end{xy}
 \hspace{15mm}
  \begin{xy}
   (10,0)*{\widetilde{W}=a_1a_2a_3+b_1b_2b_3+c_1c_2c_3}, +(0,-5)*{-\alpha a_1b_1c_1-\beta b_3a_2-\gamma c_3b_2}
  \end{xy}
\]
 We display all cuts of $(\widetilde{Q},\widetilde{W})$:
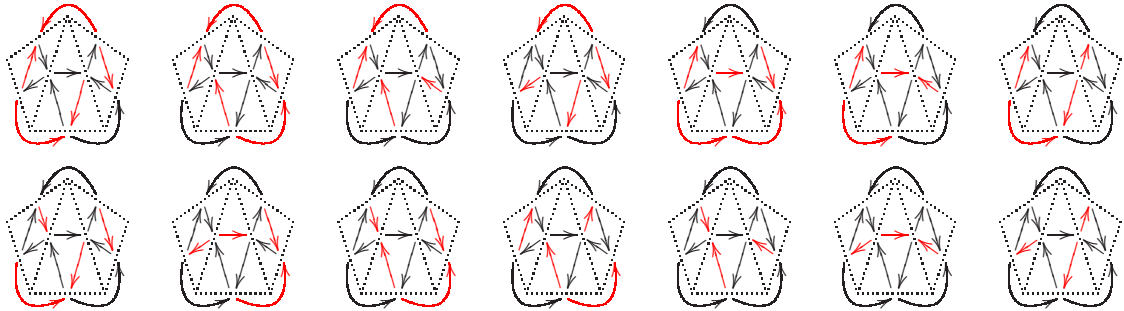
\[
  \begin{xy}
   (0,0)="0", +(8,5.75)="1", +(8,-5.75)="2", +(-3.05,-9.5)="3", +(-9.9,0)="4"
   \ar@{.}"0";"1" \ar@{.}"1";"2" \ar@{.}"2";"3" \ar@{.}"3";"4" \ar@{.}"4";"0" \ar@{.}"1";"4" \ar@{.}"1";"3" \ar@{->}"1"+(-3.75,-3.75);"1"+(-2.75,-7) \ar@{->}"1"+(-1.75,-7.5);"1"+(1.75,-7.5) \ar@{->}"1"+(2.75,-7);"1"+(3.75,-3.75) \ar@{->}"1"+(-3.15,-8.15);"1"+(-5.75,-10) {\color{red}\ar@{->}"1"+(-6,-9.5);"1"+(-4.25,-4)}{\color{red}\ar@{->}"1"+(2.25,-8.5);"1"+(0.5,-14.5)}\ar@{->}"1"+(-0.5,-14.5);"1"+(-2.25,-8.5) {\color{red}\ar@{->}"1"+(4.25,-4);"1"+(6,-9.5)}\ar@{->}"1"+(5.75,-10);"1"+(3.15,-8.15) {\color{red}\ar@{->}@/^-3.5mm/"1"+(3.75,-2);"1"+(-3.5,-2)}{\color{red}\ar@{->}@/^-3.5mm/"1"+(-6.75,-11.25);"1"+(-0.5,-16)}\ar@{->}@/^-3.5mm/"1"+(0.5,-16);"1"+(6.75,-11.25)
  \end{xy}\hspace{6mm}
  \begin{xy}
   (0,0)="0", +(8,5.75)="1", +(8,-5.75)="2", +(-3.05,-9.5)="3", +(-9.9,0)="4"
   \ar@{.}"0";"1" \ar@{.}"1";"2" \ar@{.}"2";"3" \ar@{.}"3";"4" \ar@{.}"4";"0" \ar@{.}"1";"4" \ar@{.}"1";"3"
   \ar@{->}"1"+(-3.75,-3.75);"1"+(-2.75,-7) \ar@{->}"1"+(-1.75,-7.5);"1"+(1.75,-7.5) \ar@{->}"1"+(2.75,-7);"1"+(3.75,-3.75) \ar@{->}"1"+(-3.15,-8.15);"1"+(-5.75,-10) {\color{red}\ar@{->}"1"+(-6,-9.5);"1"+(-4.25,-4)}\ar@{->}"1"+(2.25,-8.5);"1"+(0.5,-14.5) {\color{red}\ar@{->}"1"+(-0.5,-14.5);"1"+(-2.25,-8.5)}{\color{red}\ar@{->}"1"+(4.25,-4);"1"+(6,-9.5)}\ar@{->}"1"+(5.75,-10);"1"+(3.15,-8.15) {\color{red}\ar@{->}@/^-3.5mm/"1"+(3.75,-2);"1"+(-3.5,-2)}\ar@{->}@/^-3.5mm/"1"+(-6.75,-11.25);"1"+(-0.5,-16) {\color{red}\ar@{->}@/^-3.5mm/"1"+(0.5,-16);"1"+(6.75,-11.25)}
  \end{xy}\hspace{6mm}
  \begin{xy}
   (0,0)="0", +(8,5.75)="1", +(8,-5.75)="2", +(-3.05,-9.5)="3", +(-9.9,0)="4"
   \ar@{.}"0";"1" \ar@{.}"1";"2" \ar@{.}"2";"3" \ar@{.}"3";"4" \ar@{.}"4";"0" \ar@{.}"1";"4" \ar@{.}"1";"3"
   \ar@{->}"1"+(-3.75,-3.75);"1"+(-2.75,-7) \ar@{->}"1"+(-1.75,-7.5);"1"+(1.75,-7.5) \ar@{->}"1"+(2.75,-7);"1"+(3.75,-3.75) \ar@{->}"1"+(-3.15,-8.15);"1"+(-5.75,-10) {\color{red}\ar@{->}"1"+(-6,-9.5);"1"+(-4.25,-4)}\ar@{->}"1"+(2.25,-8.5);"1"+(0.5,-14.5){\color{red}\ar@{->}"1"+(-0.5,-14.5);"1"+(-2.25,-8.5)}\ar@{->}"1"+(4.25,-4);"1"+(6,-9.5) {\color{red}\ar@{->}"1"+(5.75,-10);"1"+(3.15,-8.15)}{\color{red}\ar@{->}@/^-3.5mm/"1"+(3.75,-2);"1"+(-3.5,-2)}\ar@{->}@/^-3.5mm/"1"+(-6.75,-11.25);"1"+(-0.5,-16) \ar@{->}@/^-3.5mm/"1"+(0.5,-16);"1"+(6.75,-11.25)
  \end{xy}\hspace{6mm}
  \begin{xy}
   (0,0)="0", +(8,5.75)="1", +(8,-5.75)="2", +(-3.05,-9.5)="3", +(-9.9,0)="4"
   \ar@{.}"0";"1" \ar@{.}"1";"2" \ar@{.}"2";"3" \ar@{.}"3";"4" \ar@{.}"4";"0" \ar@{.}"1";"4" \ar@{.}"1";"3"
   \ar@{->}"1"+(-3.75,-3.75);"1"+(-2.75,-7) \ar@{->}"1"+(-1.75,-7.5);"1"+(1.75,-7.5) \ar@{->}"1"+(2.75,-7);"1"+(3.75,-3.75) {\color{red}\ar@{->}"1"+(-3.15,-8.15);"1"+(-5.75,-10)}\ar@{->}"1"+(-6,-9.5);"1"+(-4.25,-4) {\color{red}\ar@{->}"1"+(2.25,-8.5);"1"+(0.5,-14.5)}\ar@{->}"1"+(-0.5,-14.5);"1"+(-2.25,-8.5) {\color{red}\ar@{->}"1"+(4.25,-4);"1"+(6,-9.5)}\ar@{->}"1"+(5.75,-10);"1"+(3.15,-8.15) {\color{red}\ar@{->}@/^-3.5mm/"1"+(3.75,-2);"1"+(-3.5,-2)}\ar@{->}@/^-3.5mm/"1"+(-6.75,-11.25);"1"+(-0.5,-16) \ar@{->}@/^-3.5mm/"1"+(0.5,-16);"1"+(6.75,-11.25)
  \end{xy}
  \hspace{6mm}
  \begin{xy}
   (0,0)="0", +(8,5.75)="1", +(8,-5.75)="2", +(-3.05,-9.5)="3", +(-9.9,0)="4"
   \ar@{.}"0";"1" \ar@{.}"1";"2" \ar@{.}"2";"3" \ar@{.}"3";"4" \ar@{.}"4";"0" \ar@{.}"1";"4" \ar@{.}"1";"3"
 \ar@{->}"1"+(-3.75,-3.75);"1"+(-2.75,-7) {\color{red}\ar@{->}"1"+(-1.75,-7.5);"1"+(1.75,-7.5)}\ar@{->}"1"+(2.75,-7);"1"+(3.75,-3.75) \ar@{->}"1"+(-3.15,-8.15);"1"+(-5.75,-10) {\color{red}\ar@{->}"1"+(-6,-9.5);"1"+(-4.25,-4)}\ar@{->}"1"+(2.25,-8.5);"1"+(0.5,-14.5) \ar@{->}"1"+(-0.5,-14.5);"1"+(-2.25,-8.5) {\color{red}\ar@{->}"1"+(4.25,-4);"1"+(6,-9.5)}\ar@{->}"1"+(5.75,-10);"1"+(3.15,-8.15) \ar@{->}@/^-3.5mm/"1"+(3.75,-2);"1"+(-3.5,-2) {\color{red}\ar@{->}@/^-3.5mm/"1"+(-6.75,-11.25);"1"+(-0.5,-16)}{\color{red}\ar@{->}@/^-3.5mm/"1"+(0.5,-16);"1"+(6.75,-11.25)}
  \end{xy}
  \hspace{6mm}
  \begin{xy}
   (0,0)="0", +(8,5.75)="1", +(8,-5.75)="2", +(-3.05,-9.5)="3", +(-9.9,0)="4"
   \ar@{.}"0";"1" \ar@{.}"1";"2" \ar@{.}"2";"3" \ar@{.}"3";"4" \ar@{.}"4";"0" \ar@{.}"1";"4" \ar@{.}"1";"3"
   \ar@{->}"1"+(-3.75,-3.75);"1"+(-2.75,-7) {\color{red}\ar@{->}"1"+(-1.75,-7.5);"1"+(1.75,-7.5)}\ar@{->}"1"+(2.75,-7);"1"+(3.75,-3.75) \ar@{->}"1"+(-3.15,-8.15);"1"+(-5.75,-10) {\color{red}\ar@{->}"1"+(-6,-9.5);"1"+(-4.25,-4)}\ar@{->}"1"+(2.25,-8.5);"1"+(0.5,-14.5) \ar@{->}"1"+(-0.5,-14.5);"1"+(-2.25,-8.5) \ar@{->}"1"+(4.25,-4);"1"+(6,-9.5) {\color{red}\ar@{->}"1"+(5.75,-10);"1"+(3.15,-8.15)}\ar@{->}@/^-3.5mm/"1"+(3.75,-2);"1"+(-3.5,-2) {\color{red}\ar@{->}@/^-3.5mm/"1"+(-6.75,-11.25);"1"+(-0.5,-16)}\ar@{->}@/^-3.5mm/"1"+(0.5,-16);"1"+(6.75,-11.25)
  \end{xy}
  \hspace{6mm}
  \begin{xy}
   (0,0)="0", +(8,5.75)="1", +(8,-5.75)="2", +(-3.05,-9.5)="3", +(-9.9,0)="4"
   \ar@{.}"0";"1" \ar@{.}"1";"2" \ar@{.}"2";"3" \ar@{.}"3";"4" \ar@{.}"4";"0" \ar@{.}"1";"4" \ar@{.}"1";"3"
   \ar@{->}"1"+(-3.75,-3.75);"1"+(-2.75,-7) \ar@{->}"1"+(-1.75,-7.5);"1"+(1.75,-7.5) {\color{red}\ar@{->}"1"+(2.75,-7);"1"+(3.75,-3.75)}\ar@{->}"1"+(-3.15,-8.15);"1"+(-5.75,-10) {\color{red}\ar@{->}"1"+(-6,-9.5);"1"+(-4.25,-4)}{\color{red}\ar@{->}"1"+(2.25,-8.5);"1"+(0.5,-14.5)}\ar@{->}"1"+(-0.5,-14.5);"1"+(-2.25,-8.5) \ar@{->}"1"+(4.25,-4);"1"+(6,-9.5) \ar@{->}"1"+(5.75,-10);"1"+(3.15,-8.15) \ar@{->}@/^-3.5mm/"1"+(3.75,-2);"1"+(-3.5,-2) {\color{red}\ar@{->}@/^-3.5mm/"1"+(-6.75,-11.25);"1"+(-0.5,-16)}\ar@{->}@/^-3.5mm/"1"+(0.5,-16);"1"+(6.75,-11.25)
  \end{xy}
\]

\[
  \begin{xy}
   (0,0)="0", +(8,5.75)="1", +(8,-5.75)="2", +(-3.05,-9.5)="3", +(-9.9,0)="4"
   \ar@{.}"0";"1" \ar@{.}"1";"2" \ar@{.}"2";"3" \ar@{.}"3";"4" \ar@{.}"4";"0" \ar@{.}"1";"4" \ar@{.}"1";"3"
   {\color{red}\ar@{->}"1"+(-3.75,-3.75);"1"+(-2.75,-7)}\ar@{->}"1"+(-1.75,-7.5);"1"+(1.75,-7.5) \ar@{->}"1"+(2.75,-7);"1"+(3.75,-3.75) \ar@{->}"1"+(-3.15,-8.15);"1"+(-5.75,-10) \ar@{->}"1"+(-6,-9.5);"1"+(-4.25,-4) {\color{red}\ar@{->}"1"+(2.25,-8.5);"1"+(0.5,-14.5)}\ar@{->}"1"+(-0.5,-14.5);"1"+(-2.25,-8.5) {\color{red}\ar@{->}"1"+(4.25,-4);"1"+(6,-9.5)}\ar@{->}"1"+(5.75,-10);"1"+(3.15,-8.15) \ar@{->}@/^-3.5mm/"1"+(3.75,-2);"1"+(-3.5,-2) {\color{red}\ar@{->}@/^-3.5mm/"1"+(-6.75,-11.25);"1"+(-0.5,-16)}\ar@{->}@/^-3.5mm/"1"+(0.5,-16);"1"+(6.75,-11.25)
  \end{xy}
  \hspace{6mm}
  \begin{xy}
   (0,0)="0", +(8,5.75)="1", +(8,-5.75)="2", +(-3.05,-9.5)="3", +(-9.9,0)="4"
   \ar@{.}"0";"1" \ar@{.}"1";"2" \ar@{.}"2";"3" \ar@{.}"3";"4" \ar@{.}"4";"0" \ar@{.}"1";"4" \ar@{.}"1";"3"
   \ar@{->}"1"+(-3.75,-3.75);"1"+(-2.75,-7) {\color{red}\ar@{->}"1"+(-1.75,-7.5);"1"+(1.75,-7.5)}\ar@{->}"1"+(2.75,-7);"1"+(3.75,-3.75) {\color{red}\ar@{->}"1"+(-3.15,-8.15);"1"+(-5.75,-10)}\ar@{->}"1"+(-6,-9.5);"1"+(-4.25,-4) \ar@{->}"1"+(2.25,-8.5);"1"+(0.5,-14.5) \ar@{->}"1"+(-0.5,-14.5);"1"+(-2.25,-8.5) {\color{red}\ar@{->}"1"+(4.25,-4);"1"+(6,-9.5)}\ar@{->}"1"+(5.75,-10);"1"+(3.15,-8.15) \ar@{->}@/^-3.5mm/"1"+(3.75,-2);"1"+(-3.5,-2) \ar@{->}@/^-3.5mm/"1"+(-6.75,-11.25);"1"+(-0.5,-16) {\color{red}\ar@{->}@/^-3.5mm/"1"+(0.5,-16);"1"+(6.75,-11.25)}
  \end{xy}
  \hspace{6mm}
  \begin{xy}
   (0,0)="0", +(8,5.75)="1", +(8,-5.75)="2", +(-3.05,-9.5)="3", +(-9.9,0)="4"
   \ar@{.}"0";"1" \ar@{.}"1";"2" \ar@{.}"2";"3" \ar@{.}"3";"4" \ar@{.}"4";"0" \ar@{.}"1";"4" \ar@{.}"1";"3"
   {\color{red}\ar@{->}"1"+(-3.75,-3.75);"1"+(-2.75,-7)}\ar@{->}"1"+(-1.75,-7.5);"1"+(1.75,-7.5) \ar@{->}"1"+(2.75,-7);"1"+(3.75,-3.75) \ar@{->}"1"+(-3.15,-8.15);"1"+(-5.75,-10) \ar@{->}"1"+(-6,-9.5);"1"+(-4.25,-4) \ar@{->}"1"+(2.25,-8.5);"1"+(0.5,-14.5) {\color{red}\ar@{->}"1"+(-0.5,-14.5);"1"+(-2.25,-8.5)}{\color{red}\ar@{->}"1"+(4.25,-4);"1"+(6,-9.5)}\ar@{->}"1"+(5.75,-10);"1"+(3.15,-8.15) \ar@{->}@/^-3.5mm/"1"+(3.75,-2);"1"+(-3.5,-2) \ar@{->}@/^-3.5mm/"1"+(-6.75,-11.25);"1"+(-0.5,-16) {\color{red}\ar@{->}@/^-3.5mm/"1"+(0.5,-16);"1"+(6.75,-11.25)}
  \end{xy}
  \hspace{6mm}
  \begin{xy}
   (0,0)="0", +(8,5.75)="1", +(8,-5.75)="2", +(-3.05,-9.5)="3", +(-9.9,0)="4"
   \ar@{.}"0";"1" \ar@{.}"1";"2" \ar@{.}"2";"3" \ar@{.}"3";"4" \ar@{.}"4";"0" \ar@{.}"1";"4" \ar@{.}"1";"3"
   \ar@{->}"1"+(-3.75,-3.75);"1"+(-2.75,-7) \ar@{->}"1"+(-1.75,-7.5);"1"+(1.75,-7.5) {\color{red}\ar@{->}"1"+(2.75,-7);"1"+(3.75,-3.75)}\ar@{->}"1"+(-3.15,-8.15);"1"+(-5.75,-10) {\color{red}\ar@{->}"1"+(-6,-9.5);"1"+(-4.25,-4)}\ar@{->}"1"+(2.25,-8.5);"1"+(0.5,-14.5) {\color{red}\ar@{->}"1"+(-0.5,-14.5);"1"+(-2.25,-8.5)}\ar@{->}"1"+(4.25,-4);"1"+(6,-9.5) \ar@{->}"1"+(5.75,-10);"1"+(3.15,-8.15) \ar@{->}@/^-3.5mm/"1"+(3.75,-2);"1"+(-3.5,-2) \ar@{->}@/^-3.5mm/"1"+(-6.75,-11.25);"1"+(-0.5,-16) {\color{red}\ar@{->}@/^-3.5mm/"1"+(0.5,-16);"1"+(6.75,-11.25)}
  \end{xy}
  \hspace{6mm}
  \begin{xy}
   (0,0)="0", +(8,5.75)="1", +(8,-5.75)="2", +(-3.05,-9.5)="3", +(-9.9,0)="4"
   \ar@{.}"0";"1" \ar@{.}"1";"2" \ar@{.}"2";"3" \ar@{.}"3";"4" \ar@{.}"4";"0" \ar@{.}"1";"4" \ar@{.}"1";"3"
   {\color{red}\ar@{->}"1"+(-3.75,-3.75);"1"+(-2.75,-7)}\ar@{->}"1"+(-1.75,-7.5);"1"+(1.75,-7.5) \ar@{->}"1"+(2.75,-7);"1"+(3.75,-3.75) \ar@{->}"1"+(-3.15,-8.15);"1"+(-5.75,-10) \ar@{->}"1"+(-6,-9.5);"1"+(-4.25,-4) \ar@{->}"1"+(2.25,-8.5);"1"+(0.5,-14.5) {\color{red}\ar@{->}"1"+(-0.5,-14.5);"1"+(-2.25,-8.5)}\ar@{->}"1"+(4.25,-4);"1"+(6,-9.5) {\color{red}\ar@{->}"1"+(5.75,-10);"1"+(3.15,-8.15)}\ar@{->}@/^-3.5mm/"1"+(3.75,-2);"1"+(-3.5,-2) \ar@{->}@/^-3.5mm/"1"+(-6.75,-11.25);"1"+(-0.5,-16) \ar@{->}@/^-3.5mm/"1"+(0.5,-16);"1"+(6.75,-11.25)
  \end{xy}
  \hspace{6mm}
  \begin{xy}
   (0,0)="0", +(8,5.75)="1", +(8,-5.75)="2", +(-3.05,-9.5)="3", +(-9.9,0)="4"
   \ar@{.}"0";"1" \ar@{.}"1";"2" \ar@{.}"2";"3" \ar@{.}"3";"4" \ar@{.}"4";"0" \ar@{.}"1";"4" \ar@{.}"1";"3"
   \ar@{->}"1"+(-3.75,-3.75);"1"+(-2.75,-7) {\color{red}\ar@{->}"1"+(-1.75,-7.5);"1"+(1.75,-7.5)}\ar@{->}"1"+(2.75,-7);"1"+(3.75,-3.75) {\color{red}\ar@{->}"1"+(-3.15,-8.15);"1"+(-5.75,-10)}\ar@{->}"1"+(-6,-9.5);"1"+(-4.25,-4) \ar@{->}"1"+(2.25,-8.5);"1"+(0.5,-14.5) \ar@{->}"1"+(-0.5,-14.5);"1"+(-2.25,-8.5) \ar@{->}"1"+(4.25,-4);"1"+(6,-9.5) {\color{red}\ar@{->}"1"+(5.75,-10);"1"+(3.15,-8.15)}\ar@{->}@/^-3.5mm/"1"+(3.75,-2);"1"+(-3.5,-2) \ar@{->}@/^-3.5mm/"1"+(-6.75,-11.25);"1"+(-0.5,-16) \ar@{->}@/^-3.5mm/"1"+(0.5,-16);"1"+(6.75,-11.25)
  \end{xy}
  \hspace{6mm}
  \begin{xy}
   (0,0)="0", +(8,5.75)="1", +(8,-5.75)="2", +(-3.05,-9.5)="3", +(-9.9,0)="4"
   \ar@{.}"0";"1" \ar@{.}"1";"2" \ar@{.}"2";"3" \ar@{.}"3";"4" \ar@{.}"4";"0" \ar@{.}"1";"4" \ar@{.}"1";"3"
   \ar@{->}"1"+(-3.75,-3.75);"1"+(-2.75,-7) \ar@{->}"1"+(-1.75,-7.5);"1"+(1.75,-7.5) {\color{red}\ar@{->}"1"+(2.75,-7);"1"+(3.75,-3.75)}{\color{red}\ar@{->}"1"+(-3.15,-8.15);"1"+(-5.75,-10)}\ar@{->}"1"+(-6,-9.5);"1"+(-4.25,-4) {\color{red}\ar@{->}"1"+(2.25,-8.5);"1"+(0.5,-14.5)}\ar@{->}"1"+(-0.5,-14.5);"1"+(-2.25,-8.5) \ar@{->}"1"+(4.25,-4);"1"+(6,-9.5) \ar@{->}"1"+(5.75,-10);"1"+(3.15,-8.15) \ar@{->}@/^-3.5mm/"1"+(3.75,-2);"1"+(-3.5,-2) \ar@{->}@/^-3.5mm/"1"+(-6.75,-11.25);"1"+(-0.5,-16) \ar@{->}@/^-3.5mm/"1"+(0.5,-16);"1"+(6.75,-11.25)
  \end{xy}
\]
\end{example}\vspace{3mm}

 Cuts of $(\widetilde{Q},\widetilde{W})$ have the following property, where we denote by $|C|$ the cardinality of a cut $C$.

\begin{lemma}\label{minnum}
 (a) For any cut $C$ of $(\widetilde{Q},\widetilde{W})$, we have $|C| \ge n+1 = |Q_0|+1$.\par
 (b) The equality in (a) holds if and only if $C$ does not contain any external arrow. In this case, $C \subset \overline{Q}_1$.
\end{lemma}

\begin{proof}
 There are $n+1$ triangle cycles not sharing arrows with each other, and triangle cycles does not contain external arrows. Thus (a) and the first assertion of (b) follow. Since there are $n+1$ big cycles, the second assertion of (b) follows. 
\end{proof}


\subsection{Maximal discrete subsets and minimal cuts} We define minimal cuts.

\begin{definition}\label{mincut}
A cut $C$ of $(\widetilde{Q},\widetilde{W})$ is called {\it minimal} if $|C| = n+1$.
\end{definition}

By Theorem \ref{cut} below, $(\widetilde{Q},\widetilde{W})$ always has minimal cuts.

\begin{example}
 In Example \ref{A2cut}, there are three minimal cuts:
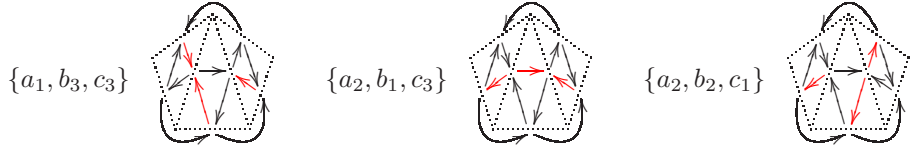
\[
  \{a_1,b_3,c_3\}\hspace{3mm}
  \begin{xy}
   (0,3)="0", +(8,5.75)="1", +(8,-5.75)="2", +(-3.05,-9.5)="3", +(-9.9,0)="4"
   \ar@{.}"0";"1" \ar@{.}"1";"2" \ar@{.}"2";"3" \ar@{.}"3";"4" \ar@{.}"4";"0" \ar@{.}"1";"4" \ar@{.}"1";"3"
   {\color{red}\ar@{->}"1"+(-3.75,-3.75);"1"+(-2.75,-7)}\ar@{->}"1"+(-1.75,-7.5);"1"+(1.75,-7.5) \ar@{->}"1"+(2.75,-7);"1"+(3.75,-3.75) \ar@{->}"1"+(-3.15,-8.15);"1"+(-5.75,-10) \ar@{->}"1"+(-6,-9.5);"1"+(-4.25,-4) \ar@{->}"1"+(2.25,-8.5);"1"+(0.5,-14.5) {\color{red}\ar@{->}"1"+(-0.5,-14.5);"1"+(-2.25,-8.5)}\ar@{->}"1"+(4.25,-4);"1"+(6,-9.5) {\color{red}\ar@{->}"1"+(5.75,-10);"1"+(3.15,-8.15)}\ar@{->}@/^-3.5mm/"1"+(3.75,-2);"1"+(-3.5,-2) \ar@{->}@/^-3.5mm/"1"+(-6.75,-11.25);"1"+(-0.5,-16) \ar@{->}@/^-3.5mm/"1"+(0.5,-16);"1"+(6.75,-11.25)
  \end{xy}
  \hspace{7mm}
  \{a_2,b_1,c_3\}\hspace{3mm}
  \begin{xy}
   (0,3)="0", +(8,5.75)="1", +(8,-5.75)="2", +(-3.05,-9.5)="3", +(-9.9,0)="4"
   \ar@{.}"0";"1" \ar@{.}"1";"2" \ar@{.}"2";"3" \ar@{.}"3";"4" \ar@{.}"4";"0" \ar@{.}"1";"4" \ar@{.}"1";"3"
   \ar@{->}"1"+(-3.75,-3.75);"1"+(-2.75,-7) {\color{red}\ar@{->}"1"+(-1.75,-7.5);"1"+(1.75,-7.5)}\ar@{->}"1"+(2.75,-7);"1"+(3.75,-3.75) {\color{red}\ar@{->}"1"+(-3.15,-8.15);"1"+(-5.75,-10)}\ar@{->}"1"+(-6,-9.5);"1"+(-4.25,-4) \ar@{->}"1"+(2.25,-8.5);"1"+(0.5,-14.5) \ar@{->}"1"+(-0.5,-14.5);"1"+(-2.25,-8.5) \ar@{->}"1"+(4.25,-4);"1"+(6,-9.5) {\color{red}\ar@{->}"1"+(5.75,-10);"1"+(3.15,-8.15)}\ar@{->}@/^-3.5mm/"1"+(3.75,-2);"1"+(-3.5,-2) \ar@{->}@/^-3.5mm/"1"+(-6.75,-11.25);"1"+(-0.5,-16) \ar@{->}@/^-3.5mm/"1"+(0.5,-16);"1"+(6.75,-11.25)
  \end{xy}
  \hspace{7mm}
  \{a_2,b_2,c_1\}\hspace{3mm}
  \begin{xy}
   (0,3)="0", +(8,5.75)="1", +(8,-5.75)="2", +(-3.05,-9.5)="3", +(-9.9,0)="4"
   \ar@{.}"0";"1" \ar@{.}"1";"2" \ar@{.}"2";"3" \ar@{.}"3";"4" \ar@{.}"4";"0" \ar@{.}"1";"4" \ar@{.}"1";"3"
   \ar@{->}"1"+(-3.75,-3.75);"1"+(-2.75,-7) \ar@{->}"1"+(-1.75,-7.5);"1"+(1.75,-7.5) {\color{red}\ar@{->}"1"+(2.75,-7);"1"+(3.75,-3.75)}{\color{red}\ar@{->}"1"+(-3.15,-8.15);"1"+(-5.75,-10)}\ar@{->}"1"+(-6,-9.5);"1"+(-4.25,-4) {\color{red}\ar@{->}"1"+(2.25,-8.5);"1"+(0.5,-14.5)}\ar@{->}"1"+(-0.5,-14.5);"1"+(-2.25,-8.5) \ar@{->}"1"+(4.25,-4);"1"+(6,-9.5) \ar@{->}"1"+(5.75,-10);"1"+(3.15,-8.15) \ar@{->}@/^-3.5mm/"1"+(3.75,-2);"1"+(-3.5,-2) \ar@{->}@/^-3.5mm/"1"+(-6.75,-11.25);"1"+(-0.5,-16) \ar@{->}@/^-3.5mm/"1"+(0.5,-16);"1"+(6.75,-11.25)
  \end{xy}
\]
\end{example}

 We give the main theorem of this section.

\begin{theorem}\label{cut}
 Maximal discrete subsets of $\overline{Q}$ are precisely minimal cuts of $(\widetilde{Q},\widetilde{W})$.
\end{theorem}

\begin{proof}
 Let $C$ be a subset of $\overline{Q}_1$ and $A:=\rho^{-1}(C)$. Then $A$ satisfies (A1) (resp., (A2)) if and only if any triangle cycle (resp., big cycle) contains precisely one arrow in $C$. Thus $C$ is a cut if and only if $A$ is a perfect matching. By Proposition \ref{DPM}, this is equivalent to that $C$ is a maximal discrete subset. Since minimal cuts are precisely cuts contained in $\overline{Q}_1$ by Lemma \ref{minnum} (b), the assertion follows.
\end{proof}

 We can give another cluster expansion formula in terms of minimal cuts from Corollary \ref{main1} and Theorem \ref{cut}. Let $(\widetilde{Q}^{[i,j]},\widetilde{W}^{[i,j]})$ be the quiver with potential obtained from $T^{[i,j]}$, and ${\mathbb C}(\widetilde{Q}^{[i,j]},\widetilde{W}^{[i,j]})$ be the set of all minimal cuts of $(\widetilde{Q}^{[i,j]},\widetilde{W}^{[i,j]})$.

\begin{corollary}\label{cutformula}
 For $1 \le i \le j \le n$, we have
 \begin{equation*}
  f^{[i,j]}=\sum_{C \in {\mathbb C}(\widetilde{Q}^{[i,j]},\widetilde{W}^{[i,j]})} \prod_{\alpha \in C}x_{\alpha},
 \end{equation*}
where $x_{\alpha}$ is the initial cluster variable corresponding to the third side of the triangle in $T$ with sides $s(\alpha)$ and $t(\alpha)$.
\end{corollary}

\begin{proof}
 It follows from Corollary \ref{main1} and Theorem \ref{cut}.
\end{proof}

\begin{example}
 For the quiver $Q=[ 1 \rightarrow 2 \rightarrow 3 ]$, we consider the quiver with potential $(\widetilde{Q},\widetilde{W})$.
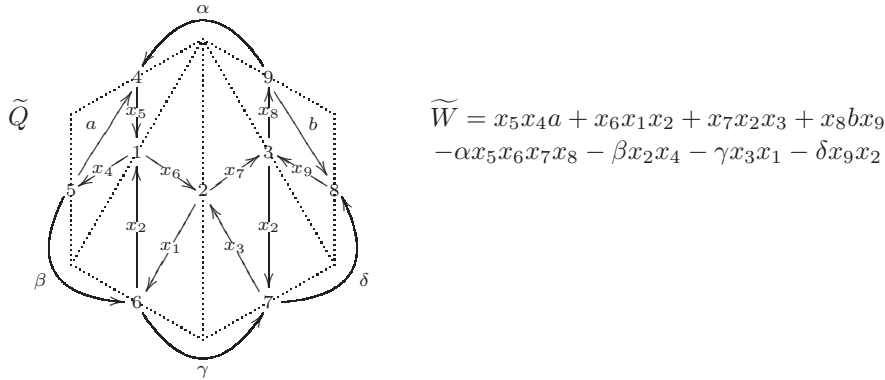
\[
  \begin{xy}
   (0,0)="0", +(17.5,10)="1", +(17.5,-10)="2", +(0,-20)="3", +(-17.5,-10)="4", +(-17.5,10)="5", "0"+(-7,0)*{\widetilde{Q}}
   \ar@{.}"0";"1"|{4} \ar@{.}"1";"2"|{9} \ar@{.}"2";"3"|{8} \ar@{.}"3";"4"|{7} \ar@{.}"4";"5"|{6} \ar@{.}"5";"0"|{5} \ar@{.}"1";"5"|{1} \ar@{.}"1";"4"|{2} \ar@{.}"1";"3"|{3} \ar@{->}"1"+(-8.8,-6.5);"1"+(-8.8,-13)|{x_5} \ar@{->}"1"+(-10,-15.5);"1"+(-16.5,-19.5)|{x_4} \ar@{->}"1"+(-16.5,-18);"1"+(-9.5,-7)^{a} \ar@{->}"1"+(-7.5,-15.5);"1"+(-1,-20)|{x_6} \ar@{->}"1"+(-1,-22);"1"+(-7.5,-33.5)|{x_1} \ar@{->}"1"+(-8.8,-33);"1"+(-8.8,-17)|{x_2} \ar@{->}"1"+(1,-20);"1"+(7.5,-15.5)|{x_7} \ar@{->}"1"+(8.8,-17);"1"+(8.8,-33)|{x_2} \ar@{->}"1"+(7.5,-33.5);"1"+(1,-22)|{x_3} \ar@{->}"1"+(8.8,-13);"1"+(8.8,-6.5)|{x_8} \ar@{->}"1"+(9.5,-7);"1"+(16.5,-18)^{b} \ar@{->}"1"+(16.5,-19.5);"1"+(10,-15.5)|{x_9} \ar@{->}@/^-6mm/"1"+(8,-3.5);"1"+(-8,-3.5)_{\alpha} \ar@{->}@/^-6mm/"1"+(-18.5,-21);"1"+(-10.5,-35)_{\beta} \ar@{->}@/^-6mm/"1"+(-8,-36.5);"1"+(8,-36.5)_{\gamma} \ar@{->}@/^-6mm/"1"+(10.5,-35);"1"+(18.5,-21)_{\delta}
  \end{xy}
  \hspace{7mm}
  \begin{xy}
   (0,0)*{\widetilde{W}=x_5x_4a+x_6x_1x_2+x_7x_2x_3+x_8bx_9} +(0,-5)*{-\alpha x_5x_6x_7x_8-\beta x_2x_4-\gamma x_3x_1-\delta x_9x_2}
  \end{xy}
\]
 We have $f^{[1,3]}=x_1x_4x_7x_9+x_3x_4x_6x_9+x_1x_2x_4x_8+x_2x_3x_5x_9$ since there are four minimal cuts of $(\widetilde{Q},\widetilde{W})$ as follows:
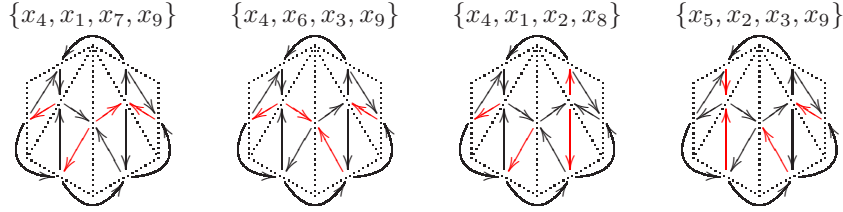
\[
  \begin{xy}
   (0,0)="0", +(8.75,5)="1", +(8.75,-5)="2", +(0,-10)="3", +(-8.75,-5)="4", +(-8.75,5)="5", "1"+(0,4)*{\{x_4,x_1,x_7,x_9\}}
   \ar@{.}"0";"1" \ar@{.}"1";"2" \ar@{.}"2";"3" \ar@{.}"3";"4" \ar@{.}"4";"5" \ar@{.}"5";"0" \ar@{.}"1";"5" \ar@{.}"1";"4" \ar@{.}"1";"3" \ar@{->}"1"+(-4.4,-3.25);"1"+(-4.4,-6.5) {\color{red}\ar@{->}"1"+(-5,-7.75);"1"+(-8.25,-9.75)}\ar@{->}"1"+(-8.25,-9);"1"+(-4.75,-3.5) \ar@{->}"1"+(-3.75,-7.75);"1"+(-0.5,-10) {\color{red}\ar@{->}"1"+(-0.5,-11);"1"+(-3.75,-16.75)}\ar@{->}"1"+(-4.4,-16.5);"1"+(-4.4,-8.5) {\color{red}\ar@{->}"1"+(0.5,-10);"1"+(3.75,-7.75)}\ar@{->}"1"+(4.4,-8.5);"1"+(4.4,-16.5) \ar@{->}"1"+(3.75,-16.75);"1"+(0.5,-11) \ar@{->}"1"+(4.4,-6.5);"1"+(4.4,-3.25) \ar@{->}"1"+(4.75,-3.5);"1"+(8.25,-9) {\color{red}\ar@{->}"1"+(8.25,-9.75);"1"+(5,-7.75)}\ar@{->}@/^-3mm/"1"+(4,-1.75);"1"+(-4,-1.75) \ar@{->}@/^-3mm/"1"+(-9.25,-10.5);"1"+(-5.25,-17.5) \ar@{->}@/^-3mm/"1"+(-4,-18.25);"1"+(4,-18.25) \ar@{->}@/^-3mm/"1"+(5.25,-17.5);"1"+(9.25,-10.5)
  \end{xy}
  \hspace{7mm}
  \begin{xy}
   (0,0)="0", +(8.75,5)="1", +(8.75,-5)="2", +(0,-10)="3", +(-8.75,-5)="4", +(-8.75,5)="5", "1"+(0,4)*{\{x_4,x_6,x_3,x_9\}}
   \ar@{.}"0";"1" \ar@{.}"1";"2" \ar@{.}"2";"3" \ar@{.}"3";"4" \ar@{.}"4";"5" \ar@{.}"5";"0" \ar@{.}"1";"5" \ar@{.}"1";"4" \ar@{.}"1";"3" \ar@{->}"1"+(-4.4,-3.25);"1"+(-4.4,-6.5) {\color{red}\ar@{->}"1"+(-5,-7.75);"1"+(-8.25,-9.75)}\ar@{->}"1"+(-8.25,-9);"1"+(-4.75,-3.5) {\color{red}\ar@{->}"1"+(-3.75,-7.75);"1"+(-0.5,-10)}\ar@{->}"1"+(-0.5,-11);"1"+(-3.75,-16.75) \ar@{->}"1"+(-4.4,-16.5);"1"+(-4.4,-8.5) \ar@{->}"1"+(0.5,-10);"1"+(3.75,-7.75) \ar@{->}"1"+(4.4,-8.5);"1"+(4.4,-16.5) {\color{red}\ar@{->}"1"+(3.75,-16.75);"1"+(0.5,-11)}\ar@{->}"1"+(4.4,-6.5);"1"+(4.4,-3.25) \ar@{->}"1"+(4.75,-3.5);"1"+(8.25,-9) {\color{red}\ar@{->}"1"+(8.25,-9.75);"1"+(5,-7.75)}\ar@{->}@/^-3mm/"1"+(4,-1.75);"1"+(-4,-1.75) \ar@{->}@/^-3mm/"1"+(-9.25,-10.5);"1"+(-5.25,-17.5) \ar@{->}@/^-3mm/"1"+(-4,-18.25);"1"+(4,-18.25) \ar@{->}@/^-3mm/"1"+(5.25,-17.5);"1"+(9.25,-10.5)
  \end{xy}
  \hspace{7mm}
  \begin{xy}
   (0,0)="0", +(8.75,5)="1", +(8.75,-5)="2", +(0,-10)="3", +(-8.75,-5)="4", +(-8.75,5)="5", "1"+(0,4)*{\{x_4,x_1,x_2,x_8\}}
   \ar@{.}"0";"1" \ar@{.}"1";"2" \ar@{.}"2";"3" \ar@{.}"3";"4" \ar@{.}"4";"5" \ar@{.}"5";"0" \ar@{.}"1";"5" \ar@{.}"1";"4" \ar@{.}"1";"3" \ar@{->}"1"+(-4.4,-3.25);"1"+(-4.4,-6.5) {\color{red}\ar@{->}"1"+(-5,-7.75);"1"+(-8.25,-9.75)}\ar@{->}"1"+(-8.25,-9);"1"+(-4.75,-3.5) \ar@{->}"1"+(-3.75,-7.75);"1"+(-0.5,-10) {\color{red}\ar@{->}"1"+(-0.5,-11);"1"+(-3.75,-16.75)}\ar@{->}"1"+(-4.4,-16.5);"1"+(-4.4,-8.5) \ar@{->}"1"+(0.5,-10);"1"+(3.75,-7.75) {\color{red}\ar@{->}"1"+(4.4,-8.5);"1"+(4.4,-16.5)}\ar@{->}"1"+(3.75,-16.75);"1"+(0.5,-11) {\color{red}\ar@{->}"1"+(4.4,-6.5);"1"+(4.4,-3.25)}\ar@{->}"1"+(4.75,-3.5);"1"+(8.25,-9) \ar@{->}"1"+(8.25,-9.75);"1"+(5,-7.75) \ar@{->}@/^-3mm/"1"+(4,-1.75);"1"+(-4,-1.75) \ar@{->}@/^-3mm/"1"+(-9.25,-10.5);"1"+(-5.25,-17.5) \ar@{->}@/^-3mm/"1"+(-4,-18.25);"1"+(4,-18.25) \ar@{->}@/^-3mm/"1"+(5.25,-17.5);"1"+(9.25,-10.5)
  \end{xy}
  \hspace{7mm}
  \begin{xy}
   (0,0)="0", +(8.75,5)="1", +(8.75,-5)="2", +(0,-10)="3", +(-8.75,-5)="4", +(-8.75,5)="5", "1"+(0,4)*{\{x_5,x_2,x_3,x_9\}}
   \ar@{.}"0";"1" \ar@{.}"1";"2" \ar@{.}"2";"3" \ar@{.}"3";"4" \ar@{.}"4";"5" \ar@{.}"5";"0" \ar@{.}"1";"5" \ar@{.}"1";"4" \ar@{.}"1";"3" {\color{red}\ar@{->}"1"+(-4.4,-3.25);"1"+(-4.4,-6.5)}\ar@{->}"1"+(-5,-7.75);"1"+(-8.25,-9.75) \ar@{->}"1"+(-8.25,-9);"1"+(-4.75,-3.5) \ar@{->}"1"+(-3.75,-7.75);"1"+(-0.5,-10) \ar@{->}"1"+(-0.5,-11);"1"+(-3.75,-16.75) {\color{red}\ar@{->}"1"+(-4.4,-16.5);"1"+(-4.4,-8.5)}\ar@{->}"1"+(0.5,-10);"1"+(3.75,-7.75) \ar@{->}"1"+(4.4,-8.5);"1"+(4.4,-16.5) {\color{red}\ar@{->}"1"+(3.75,-16.75);"1"+(0.5,-11)}\ar@{->}"1"+(4.4,-6.5);"1"+(4.4,-3.25) \ar@{->}"1"+(4.75,-3.5);"1"+(8.25,-9) {\color{red}\ar@{->}"1"+(8.25,-9.75);"1"+(5,-7.75)}\ar@{->}@/^-3mm/"1"+(4,-1.75);"1"+(-4,-1.75) \ar@{->}@/^-3mm/"1"+(-9.25,-10.5);"1"+(-5.25,-17.5) \ar@{->}@/^-3mm/"1"+(-4,-18.25);"1"+(4,-18.25) \ar@{->}@/^-3mm/"1"+(5.25,-17.5);"1"+(9.25,-10.5)
  \end{xy}
\]
 Secondly, for the case $(i,j)=(1,2)$, we consider the quiver with potential $(\widetilde{Q}^{[1,2]},\widetilde{W}^{[1,2]})$.
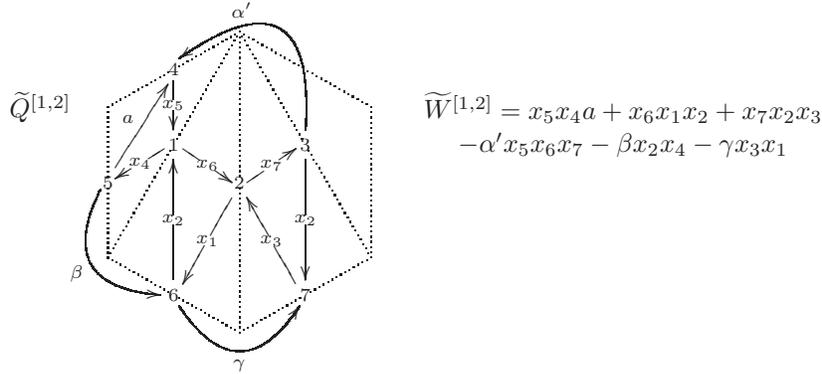
\[
  \begin{xy}
   (0,0)="0", +(17.5,10)="1", +(17.5,-10)="2", +(0,-20)="3", +(-17.5,-10)="4", +(-17.5,10)="5", "0"+(-9,0)*{\widetilde{Q}^{[1,2]}}
   \ar@{.}"0";"1"|{4} \ar@{.}"1";"2" \ar@{.}"2";"3" \ar@{.}"3";"4"|{7} \ar@{.}"4";"5"|{6} \ar@{.}"5";"0"|{5} \ar@{.}"1";"5"|{1} \ar@{.}"1";"4"|{2} \ar@{.}"1";"3"|{3} \ar@{->}"1"+(-8.8,-6.5);"1"+(-8.8,-13)|{x_5} \ar@{->}"1"+(-10,-15.5);"1"+(-16.5,-19.5)|{x_4} \ar@{->}"1"+(-16.5,-18);"1"+(-9.5,-7)^{a} \ar@{->}"1"+(-7.5,-15.5);"1"+(-1,-20)|{x_6} \ar@{->}"1"+(-1,-22);"1"+(-7.5,-33.5)|{x_1} \ar@{->}"1"+(-8.8,-33);"1"+(-8.8,-17)|{x_2} \ar@{->}"1"+(1,-20);"1"+(7.5,-15.5)|{x_7} \ar@{->}"1"+(8.8,-17);"1"+(8.8,-33)|{x_2} \ar@{->}"1"+(7.5,-33.5);"1"+(1,-22)|{x_3} \ar@{->}@/^-10mm/"1"+(8.8,-13);"1"+(-8,-3.5)_/2.5mm/{\alpha'} \ar@{->}@/^-6mm/"1"+(-18.5,-21);"1"+(-10.5,-35)_{\beta} \ar@{->}@/^-6mm/"1"+(-8,-36.5);"1"+(8,-36.5)_{\gamma}
  \end{xy}\hspace{7mm}
  \begin{xy}
   (0,0)*{\widetilde{W}^{[1,2]}=x_5x_4a+x_6x_1x_2+x_7x_2x_3} +(0,-5)*{-\alpha' x_5x_6x_7-\beta x_2x_4-\gamma x_3x_1}
  \end{xy}
\]
We have $f^{[1,2]}=x_1x_4x_7+x_3x_4x_6+x_2x_3x_5$ since there are three minimal cuts of $(\widetilde{Q}^{[1,2]},\widetilde{W}^{[1,2]})$ as follows:
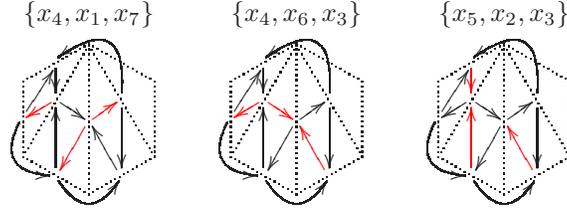
\[
  \begin{xy}
   (0,0)="0", +(8.75,5)="1", +(8.75,-5)="2", +(0,-10)="3", +(-8.75,-5)="4", +(-8.75,5)="5", "1"+(0,4)*{\{x_4,x_1,x_7\}}
   \ar@{.}"0";"1" \ar@{.}"1";"2" \ar@{.}"2";"3" \ar@{.}"3";"4" \ar@{.}"4";"5" \ar@{.}"5";"0" \ar@{.}"1";"5" \ar@{.}"1";"4" \ar@{.}"1";"3" \ar@{->}"1"+(-4.4,-3.25);"1"+(-4.4,-6.5) {\color{red}\ar@{->}"1"+(-5,-7.75);"1"+(-8.25,-9.75)}\ar@{->}"1"+(-8.25,-9);"1"+(-4.75,-3.5) \ar@{->}"1"+(-3.75,-7.75);"1"+(-0.5,-10) {\color{red}\ar@{->}"1"+(-0.5,-11);"1"+(-3.75,-16.75)}\ar@{->}"1"+(-4.4,-16.5);"1"+(-4.4,-8.5) {\color{red}\ar@{->}"1"+(0.5,-10);"1"+(3.75,-7.75)}\ar@{->}"1"+(4.4,-8.5);"1"+(4.4,-16.5) \ar@{->}"1"+(3.75,-16.75);"1"+(0.5,-11) \ar@{->}@/^-5mm/"1"+(4.4,-6.5);"1"+(-4,-1.75) \ar@{->}@/^-3mm/"1"+(-9.25,-10.5);"1"+(-5.25,-17.5) \ar@{->}@/^-3mm/"1"+(-4,-18.25);"1"+(4,-18.25)
  \end{xy}
  \hspace{10mm}
  \begin{xy}
   (0,0)="0", +(8.75,5)="1", +(8.75,-5)="2", +(0,-10)="3", +(-8.75,-5)="4", +(-8.75,5)="5", "1"+(0,4)*{\{x_4,x_6,x_3\}}
   \ar@{.}"0";"1" \ar@{.}"1";"2" \ar@{.}"2";"3" \ar@{.}"3";"4" \ar@{.}"4";"5" \ar@{.}"5";"0" \ar@{.}"1";"5" \ar@{.}"1";"4" \ar@{.}"1";"3" \ar@{->}"1"+(-4.4,-3.25);"1"+(-4.4,-6.5) {\color{red}\ar@{->}"1"+(-5,-7.75);"1"+(-8.25,-9.75)}\ar@{->}"1"+(-8.25,-9);"1"+(-4.75,-3.5) {\color{red}\ar@{->}"1"+(-3.75,-7.75);"1"+(-0.5,-10)}\ar@{->}"1"+(-0.5,-11);"1"+(-3.75,-16.75) \ar@{->}"1"+(-4.4,-16.5);"1"+(-4.4,-8.5) \ar@{->}"1"+(0.5,-10);"1"+(3.75,-7.75) \ar@{->}"1"+(4.4,-8.5);"1"+(4.4,-16.5) {\color{red}\ar@{->}"1"+(3.75,-16.75);"1"+(0.5,-11)}\ar@{->}@/^-5mm/"1"+(4.4,-6.5);"1"+(-4,-1.75) \ar@{->}@/^-3mm/"1"+(-9.25,-10.5);"1"+(-5.25,-17.5) \ar@{->}@/^-3mm/"1"+(-4,-18.25);"1"+(4,-18.25)
  \end{xy}
  \hspace{10mm}
  \begin{xy}
   (0,0)="0", +(8.75,5)="1", +(8.75,-5)="2", +(0,-10)="3", +(-8.75,-5)="4", +(-8.75,5)="5", "1"+(0,4)*{\{x_5,x_2,x_3\}}
   \ar@{.}"0";"1" \ar@{.}"1";"2" \ar@{.}"2";"3" \ar@{.}"3";"4" \ar@{.}"4";"5" \ar@{.}"5";"0" \ar@{.}"1";"5" \ar@{.}"1";"4" \ar@{.}"1";"3" {\color{red}\ar@{->}"1"+(-4.4,-3.25);"1"+(-4.4,-6.5)}\ar@{->}"1"+(-5,-7.75);"1"+(-8.25,-9.75) \ar@{->}"1"+(-8.25,-9);"1"+(-4.75,-3.5) \ar@{->}"1"+(-3.75,-7.75);"1"+(-0.5,-10) \ar@{->}"1"+(-0.5,-11);"1"+(-3.75,-16.75) {\color{red}\ar@{->}"1"+(-4.4,-16.5);"1"+(-4.4,-8.5)}\ar@{->}"1"+(0.5,-10);"1"+(3.75,-7.75) \ar@{->}"1"+(4.4,-8.5);"1"+(4.4,-16.5) {\color{red}\ar@{->}"1"+(3.75,-16.75);"1"+(0.5,-11)}\ar@{->}@/^-5mm/"1"+(4.4,-6.5);"1"+(-4,-1.75) \ar@{->}@/^-3mm/"1"+(-9.25,-10.5);"1"+(-5.25,-17.5) \ar@{->}@/^-3mm/"1"+(-4,-18.25);"1"+(4,-18.25)
  \end{xy}
\]
\end{example}

\medskip\noindent{\bf Acknowledgements}.
The author would like to thank his supervisor Osamu Iyama and Laurent Demonet for the helpful advice and instruction.

\end{document}